%% file: Liang-Rakhlin-Zhai-19.tex
\documentclass[11pt, oneside]{article}   	
\usepackage{geometry}                		
\usepackage{graphicx}
\usepackage{amssymb}
\usepackage{diagbox}
\usepackage{amsmath}
\usepackage{amsthm}
\usepackage{enumerate}
\usepackage{comment}
\usepackage{bm}
\theoremstyle{definition}

\newtheorem{prop}{Proposition}

\newtheorem{theorem}{Theorem}
\newtheorem{corollary}{Corollary}

\newtheorem{asmp}{Assumption}
\newtheorem{lemma}{Lemma}

{\theoremstyle{plain}

}

\usepackage{color}

\usepackage[usenames,dvipsnames,svgnames,table]{xcolor}
\usepackage[colorlinks=true, linkcolor=red, urlcolor=blue, citecolor=blue]{hyperref}

\usepackage[round]{natbib}
\usepackage{fullpage}

\newcommand{\En}{\mathbb E}
\newcommand{\mrm}[1]{\mathrm{#1}}
\newcommand{\argmin}[1]{\underset{#1}{\mrm{argmin}} \ }
\newcommand{\algo}{\widehat{f}}
\newcommand{\cH}{\mathcal{H}}
\newcommand{\norm}[1]{\left\|#1\right\|}
\newcommand{\tr}{\ensuremath{{\scriptscriptstyle\mathsf{\,T}}}}

\newcommand{\reals}{{\mathbb R}}

\definecolor{BrickRed}{rgb}{0.6,0,0}

\title{On the Multiple Descent of Minimum-Norm Interpolants and Restricted Lower Isometry of Kernels}
\author{Tengyuan Liang\\ University of Chicago \and Alexander Rakhlin\\ MIT \and Xiyu Zhai\\ MIT}
\date\today							

\begin{document}

\maketitle
\begin{abstract}
	We study the risk of minimum-norm interpolants of data in Reproducing Kernel Hilbert Spaces. Our upper bounds on the risk are of a \emph{multiple-descent} shape for the various scalings of $d = n^{\alpha}$, $\alpha\in(0,1)$, for the input dimension $d$ and sample size $n$. Empirical evidence supports our finding that minimum-norm interpolants in RKHS can exhibit this unusual non-monotonicity in sample size; furthermore,  locations of the peaks in our experiments match our theoretical predictions. Since gradient flow on appropriately initialized wide neural networks converges to a  minimum-norm interpolant with respect to a certain kernel, our analysis also yields novel estimation and generalization guarantees for these over-parametrized models.
	
	At the heart of our analysis is a study of spectral properties of the random kernel matrix restricted to a filtration of eigen-spaces of the population covariance operator, and may be of independent interest. 
\end{abstract}	


\input{main.tex}

\bibliography{refs}

\input{appendix.tex}

\end{document}

%% file: main.tex
\section{Introduction}

We investigate the generalization and consistency of minimum-norm interpolants
\begin{align}
	\label{def:rkhs_interpolation}
	\algo~\in~ &\argmin{f\in\cH} ~~\norm{f}_{\cH} ~~~ \text{s.t.}~~ f(x_i)=y_i,~~ i=1,\ldots,n 
\end{align}
of the data $(x_1,y_1),\ldots,(x_n,y_n) \in\reals^d\times \reals$ with respect to a norm in a Reproducing Kernel Hilbert Space $\cH$. The interpolant, also termed  ``Kernel Ridgeless Regression,'' can be viewed as a limiting solution of
\begin{align}
	\label{def:rkhs}
	\argmin{f\in\cH} \frac{1}{n}\sum_{i=1}^n (f(x_i)-y_i)^2 + \lambda\norm{f}^2_{\cH}
\end{align}
as $\lambda\to 0$. Classical statistical analyses of Kernel Ridge Regression (see e.g. \cite{caponnetto2007optimal} and references therein) rely on a carefully chosen regularization parameter to control the bias-variance tradeoff, and the question of consistency of the non-regularized solution falls outside the scope of these classical results.

Recent literature has focused on understanding risk of estimators that interpolate data, including work on nonparametric local rules \citep{belkin2018overfitting,belkin2018does}, high-dimensional linear regression \citep{bartlett2019benign,hastie2019surprises}, random features model \citep{ghorbani2019linearized}, classification with rare instances \citep{feldman2019does}, and kernel (ridgeless) regression \citep{belkin2018understand,liang2018just,rakhlin2018consistency}. 

This paper continues the line of work on kernel regression. More precisely, \cite{rakhlin2018consistency} showed that the minimum-norm interpolant with respect to Laplace kernel is \emph{not} consistent (that is, risk does not go to zero with $n\to\infty$) if dimensionality $d$ of the data is constant with respect to $n$, even if the bandwidth of the kernel is chosen adaptively. On the other hand, \cite{liang2018just} investigated the regime $n\asymp d$ and showed that risk can be upper bounded by a quantity that can be small under favorable spectral properties of the data and the kernel matrix.

The present paper aims to paint a more comprehensive picture, studying the performance of the minimum-norm interpolants in a general scaling regime $d \asymp n^{\alpha}$, $\alpha\in(0,1)$. Figure~\ref{fig:multipledescent} summarizes the non-monotone behavior of our upper bound on the risk of the minimum-norm interpolant, as reported in Theorem~\ref{thm:informal} below. 

\begin{figure}[htbp]
\centering
\includegraphics[width=0.6\linewidth]{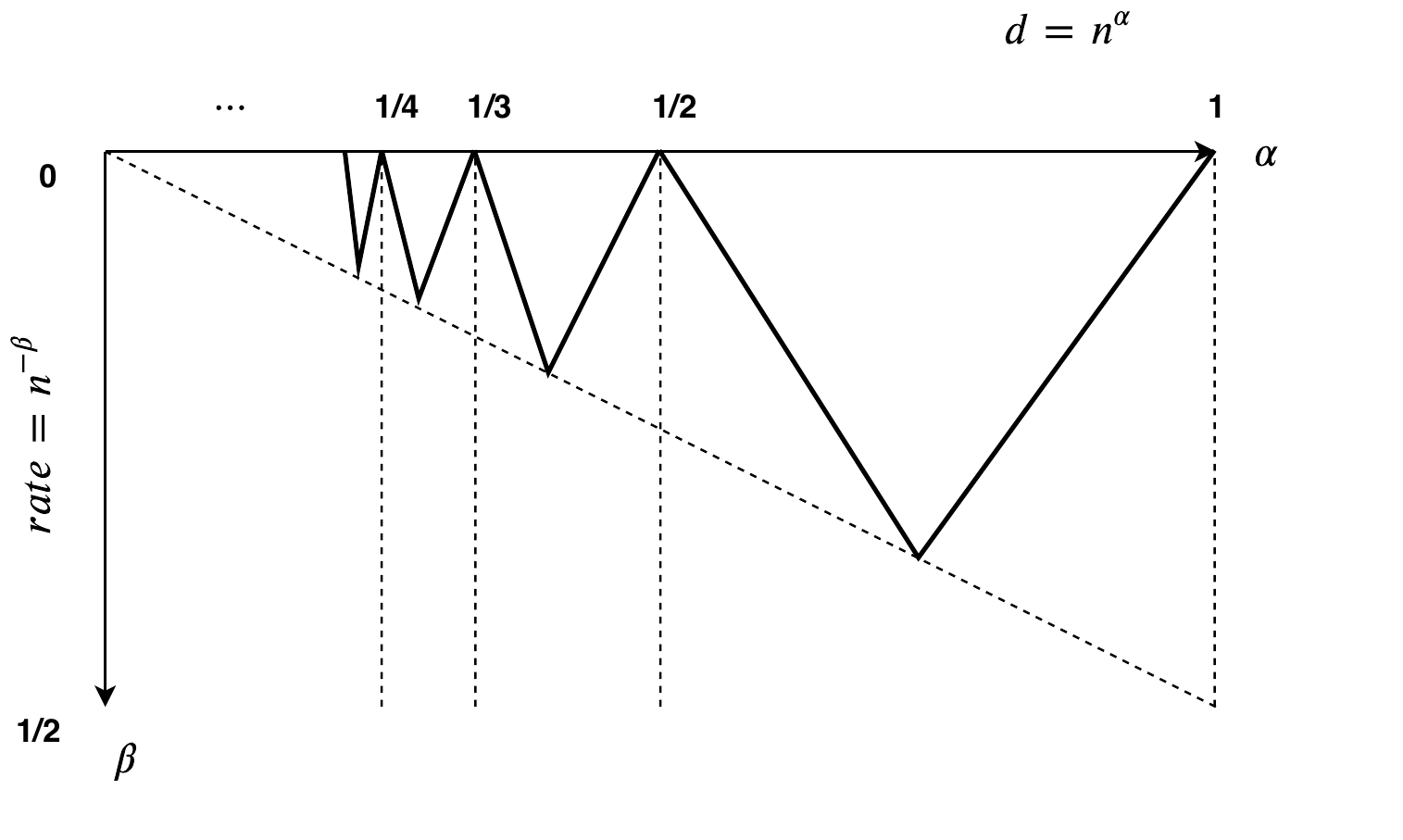}
\caption{\textbf{Multiple-descent behavior} of the rates as the scaling $d = n^{\alpha}$ changes.}
 \label{fig:multipledescent}
\end{figure}

We make two observations. First, for any integer $\iota \geq 1$, for $\alpha \in [\frac{1}{\iota+1}, \frac{1}{\iota})$, there exists a ``valley'' on the curve at each $d = n^{\frac{1}{\iota+1/2}}$ where the rate is fast (of the order $n^{-\beta}$ with $\beta = \frac{1}{2\iota+1}$). Second, towards the lower-dimensional regime ($\alpha$ moving towards $0$), the fastest possible rate even at the bottom of the valley is getting worse, with no consistency in the $\alpha=0$ regime, matching the lower bound of \citep{rakhlin2018consistency}. 

Depending on the point of view, we can also interpret the upper bound of Theorem~\ref{thm:informal} by fixing $d$ and analyzing the behavior in $n$. In this case, the interpretation is rather counterintuitive: more data can lead to alternating regimes of better and worse performance. Conceptually, this occurs because with more samples, the empirical kernel matrix could estimate more complex subspaces associated with smaller population eigenvalues, thus increasing the variance of the interpolants.

Our experiments, reported in  Figure~\ref{fig:graphics_experiment_n=5000}, confirm the surprising multiple-peaks behavior, suggesting that the non-monotone shape of our upper bound is not just an artifact of the proof technique. Moreover, the locations of the peaks line up with our theoretical predictions. Our finding complements the \emph{double-descent} behavior investigated previously in the literature \citep{belkin2018reconciling,mei2019generalization},\footnote{In Figure~\ref{fig:graphics_experiment_n=5000}, we only plotted the variance of the minimum-norm interpolant since the shape will dominate the bias term for appropriately scaled conditional variance of the $Y$ variable.} suggesting that the behavior in the kernel case is significantly more detailed.

The challenging problem of proving a \emph{lower bound} that exhibits the multiple descent behavior remains open. A more detailed analysis that studies relative heights of the peaks also appears to be an interesting direction of investigation. While the peaks and their size are certainly interesting, the reader should also note the positive message of our main result: the interpolating solution provably has a diminishing (in sample size) out-of-sample error for most of the scalings of $d$ and $n$.


\vspace{-0.1cm}
\begin{figure}[htbp]
  \centering
    \includegraphics[width=.5\textwidth]{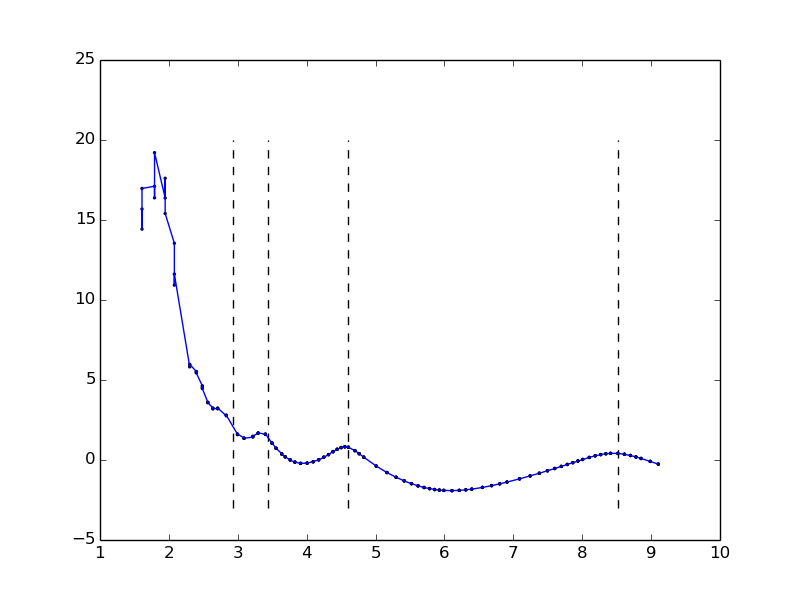}
  \caption{Empirical evidence for multiple descent. Sample size $n=5000$, x-axis: $\log d$, y-axis: variance of the minimum-norm interpolant. Vertical lines denote theoretically predicted peaks.}
  \label{fig:graphics_experiment_n=5000}
\end{figure}

The main result of the paper can be informally stated as follows.
\begin{theorem}[Informal]
	\label{thm:informal}
	For any integer $\iota \geq 1$, consider $d = n^{\alpha}$ where $\alpha \in [\frac{1}{\iota+1}, \frac{1}{\iota}).$ Consider a general function $h \in C^{\infty}(\mathbb{R})$ and define the inner product kernel $k(x, z) = h(x^\top z/d)$. Consider $n$ $i.i.d.$ data pairs $(x_i, y_i)$ drawn from $\mathcal{P}_{X, Y}$, and denote the target function $f_*(x)=\mathbb{E}[Y|X=x]$. Suppose the conditions on $f_*, h$ and $\mathcal{P}_{X, Y}$ specified by Theorems~\ref{thm:variance}-\ref{thm:bias} are satisfied.
	With probability at least $1 - \delta - e^{-n/d^{\iota}}$ on the design $\bm{X}  \in \mathbb{R}^{n\times d}$,
	\begin{align*}
		\mathbb{E}\left[ \| \widehat f - f_* \|^2_{\mathcal{P}_X} | \bm{X} \right] \leq C \cdot \left( \frac{d^{\iota}}{n}+\frac{n}{d^{\iota+1}}  \right) \asymp n^{-\beta}, \\
		\quad \beta:= \min\left\{(\iota+1)\alpha-1, 1-\iota\alpha \right\} \geq 0.
	\end{align*}
	Here the constant $C(\delta, h, \iota, \mathcal{P})$ does not depend on $d, n$, with $\mathcal{P}$ denoting the distribution of each coordinate of $X$.
\end{theorem}

It is easy to see that the minimum-norm interpolant in \eqref{def:rkhs_interpolation} has the closed-form solution 
$$\algo(x) = k(x,\bm{X})^\tr K^{-1}\bm{Y}$$
if the kernel matrix $K\in \reals^{n\times n}$ is invertible. Here $k(x,\bm{X})=[k(x,x_1),\ldots,k(x,x_n)]^\tr$, $K_{i,j}=k(x_i,x_j)$, and $\bm{Y}=[y_1,\ldots,y_n]^\tr$. As discussed below, the variance of the estimator can be upper bounded by 
$$\sigma^2_Y\cdot \En_{x,\bm{X}}\norm{k(x,\bm{X})^\tr K^{-1}}^2$$
where $\sigma^2_Y$ is a uniform upper bound on the conditional variance of $Y$ given $X$. However, in general, the smallest eigenvalue of the sample kernel matrix scales as a constant. Hence, further estimates on the variance term require a careful spectral analysis of the sample-based $K$ and the population-based $k(x,\bm{X})$. Note that the eigenvalues of the empirical kernel matrix have one-to-one correspondance to that of the empirical covariance operator. We prove that on a filtration of eigen-spaces of the covariance operator defined by the population distribution, the empirical covariance operator satisfies a certain \emph{restricted lower isometry} property.  This spectral analysis is the main technical part.

\section{Restricted Lower Isometry of High-Dimensional Kernels}

In this section, we highlight Proposition \ref{prop:key}, which establishes the \textit{Restricted Lower Isometry Property}. This property proves crucial in bounding the generalization error for the kernel ridgeless regression and, as a consequence, for randomly-initialized wide neural networks trained to convergence. The detailed proof of Proposition \ref{prop:key} is deferred to Section~\ref{sec:main_proof}.

\subsection{Setup}
Before stating the main proposition, let us introduce the formal setup and assumptions for the rest of the paper.
Random vectors $x_1,\cdots,x_n \in \mathbb{R}^d$ are drawn i.i.d. from a product distribution $\mathcal{P}_X = \mathcal{P}^{\otimes d}$, where the distribution for each coordinate $\mathcal{P}$ is independent and satisfies the following property.
\begin{asmp}[Distribution for each coordinate]
	\label{asmp:dist-coordinate}
	Assume that: (1) $\mathbb{E}_{z \sim \mathcal{P}}[z] = 0$ and for a constant $\nu > 1$, $\mathbb{P}(|z|\ge t)\le C(1+t)^{- \nu}$ holds for all $t\geq 0$.
	(2) For any set $S$ of finitely many real numbers, $\mathbb{P}(z\in S)<1$.
\end{asmp}
In addition, we require that $\forall x \in \mathcal{X} \subseteq \mathbb{R}^d$, the conditional variance is bounded by a constant: ${\rm Var}[Y|X = x] \leq \sigma^2_Y$.

Consider a function $h \in \mathcal{C}^{\infty}(\mathbb{R})$ whose Taylor expansion converges for all $t \in \mathbb{R}$
\begin{equation}
	\label{eq:h}
	h(t)=\sum_{i=0}^\infty \alpha_i t^i
\end{equation}
with all coefficients $\alpha_i\ge 0$. We define a kernel function $k(\cdot, \cdot):\mathcal{X} \times \mathcal{X} \rightarrow \mathbb{R}$ induced by $h$
\begin{align}
	\label{eq:kernel-interproduct}
	k(x, z) := h\left( \frac{x^\top z}{d} \right) .
\end{align}
Similarly, we define the \emph{normalized} finite dimensional kernel matrix $\mathbf{K} \in \mathbb{R}^{n\times n}$,
\begin{equation*}
	\mathbf{K}_{ij}:=k(x_i, x_j)/n, ~~1\leq i,j\leq n
\end{equation*}
In other words, $\mathbf{K} = k(\bm{X} , \bm{X} )/n$ with the $1/n$ normalization.

Denote the truncated polynomial and the corresponding truncated kernel matrix $\mathbf{K}^{[\iota]}$ (used only in the proof) as
\begin{equation*}
	h^{[\le\iota]}(t)=\sum_{i=0}^\iota \alpha_i t^i, ~~~ \mathbf{K}^{[\le \iota]}_{ij}:=h^{[\le\iota]}(x_i^\top x_j/d)/n.
\end{equation*}
Similarly, we denote the degree-$\iota$ component and the corresponding kernel as
\begin{equation*}
	h^{[\iota]}(t)=\alpha_\iota t^\iota\enspace, ~~~ \mathbf{K}^{[\iota]}_{ij}:=h^{[\iota]}(x_i^\top x_j/d)/n\enspace.
\end{equation*}

We are interested in the following high dimensional regime:
there is a fixed positive integer $\iota$ such that
\begin{equation}
	\iota<\frac{\log n}{\log d(n)}<\iota+1.
\end{equation}
Our investigation focuses on the regime when the dimension $d(n)$ grows with the sample size $n$, with $n$ being sufficiently large.

Finally, we use the norm $\norm{g}_{L_2(\mathcal{P}_X)}^2 = \norm{g}_{\mathcal{P}_X}^2= \int g^2 \mathcal{P}_X(dx)$ to measure the quality of the estimator. Thanks to the identity $\En\norm{\algo-f_*}^2_{\mathcal{P}_X} = \En (\algo(X)-Y)^2 - \min_{f\in\cH} \En(f(X)-Y)^2$, our results on estimation directly translate into prediction error guarantees.

\subsection{Main Technical Result}

%

Now we are ready to state the main technical contribution. We establish the restricted lower isometry property of the empirical kernel matrix on a filtration of spaces indexed by the polynomial basis with increasing degree.
\begin{prop}[Restricted Lower Isometry of High Dimensional Kernel]
	\label{prop:key}
	Let $\iota_0$ be a positive integer.
	Assume that the first $\iota_0+1$ Taylor coefficients $\alpha_0,\cdots, \alpha_{\iota_0}$ of the function $h$ (defined in \eqref{eq:h}) are positive. Assume that $d^{\iota_0}\log d=o(n)$. Let Assumption~\ref{asmp:dist-coordinate} on $\mathcal{P}$ be satisfied with $\nu > \iota_0$.

	Then there are positive constants $C,C'$ depending only on $\iota_0$, $\mathcal{P}$, and $\{\alpha_i\}_{i\leq \iota_0}$ such that for $n$ large enough, with probability at least $1-e^{-\Omega(n/d^{\iota_0})}$ the following holds:
	\begin{itemize}
		\item for any $\iota\le \iota_0$, $\mathbf{K}^{[\le \iota]}$ has $\binom{\iota+d}{\iota}$ nonzero eigenvalues, all of them larger than $C'd^{-\iota}$, and 
		\item the range of $\mathbf{K}^{[\le \iota]}$ is
	\begin{align*}
		{\rm span}\Big\{(p(x_1),\cdots,p(x_n)):p \text{ is a multivariable polynomial of degree not larger than $\iota$}\Big\}.
	\end{align*}
	\end{itemize}
\end{prop}

\subsection{Proof Outline.}

First, observe that
	\begin{align*}
		n \bm{K}_{ij} & = \sum_{\iota=0}^\infty \alpha_\iota \left( \frac{x_i^\top x_j}{d} \right)^{\iota} = \sum_{r_1,\cdots,r_d \geq 0} ~ c_{r_1 \cdots r_d} \alpha_{r_1+\cdots+r_d} p_{r_1 \cdots r_d}(x_i) p_{r_1 \cdots r_d}(x_j)/ d^{r_1+\cdots+r_d}
	\end{align*}
	with
		$c_{r_1 \cdots r_d} = \frac{(r_1+\cdots+r_d)!}{r_1!\cdots r_d!}$ and monomials $p_{r_1 \cdots r_d}(x_i) = (x_i[1])^{r_1}\cdots (x_i[d])^{r_d}$ with multi-index $r_1\cdots r_d$. The degree-bounded empirical kernel is then 
	\begin{align*}
		n \bm{K}_{ij}^{[\leq \iota]} :=  \sum_{\substack{r_1,\cdots,r_d \geq 0\\r_1+\cdots+r_d \leq \iota}}~ c_{r_1 \cdots r_d} \alpha_{r_1+\cdots+r_d} p_{r_1 \cdots r_d}(x_i) p_{r_1 \cdots r_d}(x_j)/ d^{r_1+\cdots+r_d} =\Phi \cdot \Phi^\top
	\end{align*}
	with polynomial features $\Phi \in \reals^{n \times \binom{\iota+d}{\iota}}$ of the form
	\begin{align*}
		\Phi_{i, (r_1 \cdots r_d)} = \left(c_{r_1 \cdots r_d} \alpha_{r_1+\cdots+r_d}\right)^{1/2}  p_{r_1 \cdots r_d}(x_i) / d^{(r_1+\cdots+r_d)/2} \enspace.
	\end{align*}
	The restricted lower isometry of the kernel is equivalent to establishing that all eigenvalues of the  sample covariance operator 
	\begin{align*}
		\Theta^{[\leq \iota]} := \frac{1}{n} \Phi^\top \cdot \Phi 
	\end{align*}
	are lower bounded by $d^{-\iota}$. Observe that non-zero eigenvalues of $\bm{K}^{[\leq \iota]}$ have one-to-one correspondance to that of $\Theta^{[\leq \iota]}$. 
	
	Hypothetically, if the mononomials $p_{r_1 \cdots r_d}(x)$ were orthogonal in $L^2_{\mathcal{P}_X}$, then we would have
	\begin{align*}
		\mathbb{E}\left[ \Theta^{[\leq \iota]} \right] = {\rm diag}(C(0),~ \cdots, ~C(\iota')\cdot d^{-\iota'} ,~  \cdots,~ \overbrace{C(\iota) \cdot d^{-\iota}}^{\binom{d+\iota-1}{d-1}~\text{such entries}} )
	\end{align*} 
	which would prove what we aim to establish on the smallest eigenvalue, at least in expectation. However, the orthogonality does not hold, and the monomials have a complex covariance structure that we have to tackle. To address the problem, we perform the \textit{Gram-Schimdt} process on polynomials
	$$\{1, t, t^2, \cdots \} \rightarrow \{1, q_1(t), q_2(t), \cdots \} \quad \text{$q$ orthogonal polynomial basis on $L^2_{\mathcal{P}}$},$$ 
	 and show that such basis is \textit{weakly-correlated}. Then under the new polynomial features
	\begin{align*}
		\Phi_{i, (r_1 \cdots r_d)} &\rightarrow \Psi_{i, (r_1 \cdots r_d)} = \left(c_{r_1 \cdots r_d} \alpha_{r_1+\cdots+r_d}\right)^{1/2} \prod_{j \in [d]} q_{r_j}(x_i[j]) / d^{(r_1+\cdots+r_d)/2} \\
		\Phi &= \Psi \Lambda,	\quad \Lambda \in \mathbb{R}^{\binom{\iota+d}{\iota} \times \binom{\iota+d}{\iota}} \quad \text{upper-triangular} \enspace.
	\end{align*}
	It turns out that through technical calculations, we can show that such \textit{weakly-correlated} polynomial features ensure that
	\begin{align*}
		\| \Lambda \|_{\rm op}, \| \Lambda^{-1} \|_{\rm op} \leq C(\iota)  \enspace.	
	\end{align*}
	We can now focus on studying the smallest eigenvalue of the un-correlated features since for any $u \in \mathbb{R}^{\binom{\iota+d}{\iota}}$,
	\begin{align*}
		 u^{\top} \Theta^{ [\leq \iota]} u = \frac{1}{n} \| \Phi u \|^2 = \frac{1}{n} \| \Psi \Lambda u \|^2 \geq \lambda_{\min} \left( \frac{1}{n} \Psi^\top \Psi \right) \| \Lambda u\|^2 \asymp \lambda_{\min} \left( \frac{1}{n} \Psi^\top \Psi \right) \| u \|^2 \enspace.
	\end{align*}
	The next challenge is in establishing a lower bound on the above smallest eigenvalue. Here, a naive use of standard concentration fails to provide strong high probability bounds. To see this, recall that if one wants to establish deviation bound via standard concentration like below
	\begin{align*}
		\sup_{u \in B_2^{\binom{d+\iota}{\iota}}} u^\top \left(\frac{1}{n} \Psi^\top \Psi  - \mathbb{E}\left[ \frac{1}{n} \Psi^\top \Psi  \right] \right) u \lesssim \frac{{\rm complexity}(B_2^{\binom{d+\iota}{\iota}})}{\sqrt{n}}  \enspace,
	\end{align*}
	the deviation bound will typically be larger than $d^{-\iota}$ for $\iota$ of our interest. To address this, we take the \textit{small-ball} approach, pioneered in \cite{koltchinskii2015bounding,mendelson2014learning}, which utilizes the non-negativity of the quadratic process. The intuition is as follows: due to positivity of $\langle \Psi(x_i), u \rangle^2$, the following lower bound holds
	\begin{align*}
		\| \Psi u\|^2 = \frac{1}{n} \sum_{i=1}^n \langle \Psi(x_i), u \rangle^2 \geq c_1 \mathbb{E}[\langle \Psi(x_i), u \rangle^2] \cdot \frac{1}{n}  \sum_{i=1}^n I_{\langle \Psi(x_i), u \rangle^2 \geq c_1 \mathbb{E}[\langle \Psi(X), u \rangle^2]}		\enspace.
	\end{align*}
	Suppose one can show that there exist absolute constants $c_1, c_2 >0$ such that
	\begin{align*}
		\mathbb{P}\left( \langle \Psi(x_i), u \rangle^2 \geq c_1 \mathbb{E}[\langle \Psi(X), u \rangle^2] \right) \geq c_2 \enspace,
	\end{align*}
	a condition referred to as the \textit{small-ball} property. Then it immediately follows that with probability at least $1 - \exp(-c \cdot n)$
	\begin{align*}
		\frac{1}{n}  \sum_{i=1}^n I_{\langle \Psi(x_i), u \rangle^2 \geq c_1 \mathbb{E}[\langle \Psi(X), u \rangle^2]} \geq c_2/2 \enspace.
	\end{align*}
	Now the union bound on $B_2^{\binom{d+\iota}{\iota}}$ does not affect the rate significantly since the probability control is overwhelmingly small (exponential in $n$). Last but not least, the technicality remains to verify the \textit{small-ball} property for weakly dependent polynomials via \textit{Paley-Zygmund} inequality.

We note that concurrent work of \cite{ghorbani2019linearized} also implies a similar control on the least eigenvalue under a different setting with different assumptions on the underlying data. Specifically, their result concerns the approximation error on random Fourier feature models. It could be translated to a risk bound due to the dual relationship between random features and random samples, in the case when there is no label noise $y_i = f_\star(x_i)$.

The rest of the paper is organized as follows. In Section~\ref{sec:RKHS} and \ref{sec:NN} we will apply the key Proposition~\ref{prop:key} to obtain generalization results for kernel ridgeless regression and wide neural networks, respectively. Sections~\ref{sec:variance_proof} and \ref{sec:bias_proof} will be devoted to the proofs of Theorem~\ref{thm:variance} and Theorem~\ref{thm:bias} on the variance and bias of the minimum-norm interpolant. Section~\ref{sec:main_proof} in the Appendix will focus on the main steps behind proving Proposition~\ref{prop:key}. Appendix also contains several supporting lemmas.

\section{Application to Kernel Ridgeless Regression}
\label{sec:RKHS}

	The following bias-variance decomposition holds, conditionally on $\bm{X}$:
	\begin{align}
		\En_{\bm{Y}}\norm{\algo - f_* }^2_{\mathcal{P}_X}  = \En_{\bm{Y}}\norm{\algo - \En_Y[\algo] }^2_{\mathcal{P}_X}  + \norm{\En_{\bm{Y}}[\algo] - f_* }^2_{\mathcal{P}_X} \enspace.
	\end{align}
	Here $\En_{\bm{Y}}$ denotes expectation only over the $\bm{Y}$ vector. In this section, we refer to the first term as \textit{Variance}, and the second term as (squared) \textit{Bias}. As mentioned earlier, the variance term can be upper bounded as
\begin{align*}
	\sigma^2_Y \cdot \En_{x\sim \mathcal{P}_X} \norm{k(x,\bm{X})^\tr k(\bm{X},\bm{X})^{-1}}^2,
\end{align*}
thanks to the closed-form of $\algo$. In the rest of this section we shall assume, for brevity, that $\sigma^2_Y=1$.

\subsection{Variance}
\label{sec:variance}

In this section, to control the variance term we make a stronger assumption on the tail behavior of $\mathcal{P}$.
\begin{theorem}[Variance]
	\label{thm:variance}
	Let $x\sim \mathcal{P}_X$, $\bm{X}\sim \mathcal{P}^{n\times d}$, and $\mathcal{P}$ be sub-Gaussian. 
	Consider $h \in \mathcal{C}^\infty(\mathbb{R})$, denote $h(t) = \sum_{i=0}^\infty \alpha_i t^i$ with corresponding Taylor coefficients $\{\alpha_i\}_{i=0}^\infty$. Consider $k(x,z) = h(x^T z/d)$ for $x, z \in \mathcal{X}$.
	\begin{enumerate}[(i)]
		\item Suppose that:
		\begin{itemize}
			\item $\alpha_1,\cdots, \alpha_ \iota>0$;
			\item there is $\iota' \geq 2 \iota+3$ such that $\alpha_{\iota'}>0$.
		\end{itemize}
		Assume $d^{\iota}\log d \lesssim n \lesssim d^{\iota+1}$. Then with probability at least $1- e^{- \Omega(n/d^{\iota})}$ w.r.t. $\bm{X}$,

		\begin{equation}
			\text{Variance} \leq \mathbb{E}_{x \sim \mathcal{P}_X}\|k(\bm{X}, \bm{X})^{-1}k(\bm{X},x)\|^2\le C\left(\frac{d^{\iota}}{n}+\frac{n}{d^{\iota+1}}\right) \enspace.
		\end{equation}

		\item Suppose that the Taylor expansion coefficients satisfy for some $\iota>0$:

		\begin{itemize}
			\item $\alpha_1,\cdots, \alpha_{\iota}>0$;
			\item $\forall \iota'> \iota, \alpha_{\iota'}=0$, i.e. $k$ is a polynomial kernel.
		\end{itemize}
		Assume $d^{\iota}\log d \lesssim n$.
		Then with probability at least $1- e^{- \Omega(n/d^{\iota})}$ w.r.t. $\bm{X}$,

		\begin{equation}
			\text{Variance} \leq \mathbb{E}_{x \sim \mathcal{P}_X}\|k(\bm{X}, \bm{X})^{-1}k(\bm{X},x)\|^2\le C\frac{d^{\iota}}{n}\enspace.
		\end{equation}
\end{enumerate}
\end{theorem}
The proof of the above theorem, which appears in Section~\ref{sec:variance_proof}, depends on breaking the variance into two parts depending on the polynomial degree: the first part we can further upper bounded via the key restricted lower isometry proposition on a filtration of spaces (ordered according to the degree of the polynomials), and for the second part we utilize the fact that $n\bm{K} \succeq \Omega(1) \bm{I}_{n}$.


\subsection{Bias}

In this section, we bound the bias part for the min-norm interpolated solution. In fact, we will show that, under a suitable assumption, the squared bias is upper bounded by a multiple factor of the variance term, studied in the previous section. 

\begin{theorem}[Bias]
	\label{thm:bias} 
	Assume that the target function $f_*(x) = \mathbb{E}[Y|X=x]$ can be represented as
	\begin{align}
		f_*(x) = \int_{\mathcal{X}} k(x, z) \rho_*(z) \mathcal{P}_X(dz)
	\end{align}
	with $\int_{\mathcal{X}} \rho_*^4(x) \mathcal{P}_X (dx) \leq C$ for $C>0$. Assume that $\sup_{x \in \mathcal{X}} k(x, x) \leq M$. Then we have
	\begin{align*}
		\text{Bias} &:= \mathbb{E}_{x \sim \mathcal{P}_X} \left( k(x, \bm{X}) k(\bm{X}, \bm{X})^{-1} f_*(\bm{X}) - f_*(x) \right)^2 \\
		&\leq C_1(\bm{X}) \cdot \mathbb{E}_x \left\| k(\bm{X}, \bm{X})^{-1} k(\bm{X}, x) \right\|^2 + \frac{C_2(\bm{X})}{n}
	\end{align*}
	where the scalar random variables are bounded in $\ell_2$-sense: $\mathbb{E}_{\bm{X}} [C_1(\bm{X})]^2, \mathbb{E}_{\bm{X}} [C_2(\bm{X})]^2 \precsim 1$.
\end{theorem}

We remark that $\mathbb{E}_{x} \left\| k(\bm{X}, \bm{X})^{-1} k(\bm{X}, x) \right\|^2$ is the expression for the upper bound on variance in Section~\ref{sec:variance}. The statement can be strengthened to the ``in probability'' statement, as follows
	\begin{align*}
		\text{Bias} &:= \| k(x, \bm{X}) k(\bm{X}, \bm{X})^{-1} f_*(\bm{X}) - f_*(x) \|_{\mathcal{P}_X}^2 \\
		&\leq  \frac{1}{\sqrt{\delta}} \cdot \left(\mathbb{E}_x \left\| k(\bm{X}, \bm{X})^{-1} k(\bm{X}, x) \right\|^2 \vee \frac{1}{n} \right)
	\end{align*}
	with probability $1-\delta$ on $\bm{X}$. See Proposition~\ref{prop:in-prob} for details. We emphasize that here the factor $\delta^{-1/2}$ has \textit{no dependence} on the dimension $d$. Note that one can relax the assumption of $\int_{\mathcal{X}} \rho_*^4(x) \mathcal{P}_X (dx) \leq C$ to $\int_{\mathcal{X}} \rho_*^2(x) \mathcal{P}_X (dx) \leq C$ at the cost of the factor $\delta^{-1}$ instead of $\delta^{-1/2}$ in the above statement.


\section{Applications to Wide Neural Networks}
\label{sec:NN}



Before extending the results to neural networks, we need to study a Neural-Tangent-type kernel defined below in \eqref{eq:cos}, which is slightly different from the inner-product kernel as in \eqref{eq:kernel-interproduct}. Specifically, we consider kernels of the following form:
\begin{equation}
	\label{eq:cos}
	k(x,x')=\| x\|\| x'\|\sum\limits_{i=0}^\infty \alpha_i \left( \cos\angle (x,  x') \right)^i 
\end{equation}
where $ \cos\angle ( x,  x') = x^\top  x'/\|  x\|\|  x'\|$. Suppose that $\alpha_i\ge 0$ for all $i$ and $\sup \{i: \alpha_i>0\}=\infty$. Note that the above kernel reduces to the inner-product kernel when the data lies on a fixed radius sphere $\| x \| \equiv R$.

\begin{corollary}[Generalization of Neural-Tangent-Type Kernels]
	\label{coro:ntk}
	Consider the type of kernels defined in \eqref{eq:cos}, which subsumes the Neural Tangent Kernel as a special case. Consider $n$ $i.i.d.$ data pairs $(x_i, y_i)$ drawn from $\mathcal{P}_{X, Y}$, and denote the target function $f_*(x)=\mathbb{E}[Y|X=x]$. Suppose the conditions on $f_*$ and $\mathcal{P}_{X, Y}$ specified by Theorems~\ref{thm:variance}-\ref{thm:bias} are satisfied. Consider integer $\iota$ that satisfy $d^{\iota} \log d \lesssim n \lesssim d^{\iota+1}/\log d$.
	Then the result of Theorem~\ref{thm:informal} can be generalized to such kernels: with high probability, the following upper bound on the risk holds:
	\begin{align*}
		\mathbb{E}\left[ \| \widehat f - f_* \|^2_{\mathcal{P}_X} | \bm{X} \right] \leq C \cdot \left( \frac{d^{\iota}}{n}+\frac{n \log d}{d^{\iota+1}}  \right) \enspace.
	\end{align*}
\end{corollary}

The connection between Neural Tangent Kernel (NTK) \citep{jacot2018neural} and wide neural networks is by now well-known. It follows from \citep{du2018gradient} that sufficiently wide randomly initialized (with appropriate scaling) neural networks converge to the minimum-norm interpolating solution with respect to NTK, under appropriate assumptions. This connection allows us to leverage Corollary~\ref{coro:ntk} for establishing estimation and generalization guarantees for these models.

For completeness, we show that NTK for infinitely-wide neural networks is indeed of the form \eqref{eq:cos}.
Here we consider a one-hidden-layer neural network defined as follows:
\begin{equation*}
	f(x;W,a):=\frac{1}{\sqrt{m}}\sum_{j=1}^m a_j\sigma(w_j^\top \tilde x),
\end{equation*}
where the input $x$ is a $d$-dimensional vector $W=(w_1,\cdots,w_m)$ is a $(d+1)\times m$ matrix and $a=(a_1,\cdots,a_m)$ is a $m$-dimensional vector and $\tilde x=(x^\top,\sqrt{d})^\top$.
%
The NTK $h^m$ is defined by
 \begin{equation*}
 	h^m(x,x'):=\frac{1}{m}\left(\sum\limits_{j=1}^m\sigma(w_j^\top \tilde x)\sigma(w_j^\top \tilde x')+x^\top x'\sum\limits_{j=1}^ma_i^2\sigma'(w_j^\top \tilde x)\sigma'(w_j^\top \tilde x')\right) \enspace.
 \end{equation*}
 Assume that the parameters are initialized according to i.i.d. $\mathcal{N}(0,1)$. Then the above kernel converges pointwise to the following kernel as $m\rightarrow\infty$:
\begin{equation*}
	\begin{split}
		h^\infty(x,x'):&=\mathbb{E}_{w\sim \mathcal{N}(0,I_{d+1})}\left(\sigma(w^\top \tilde x)\sigma(w^\top \tilde x')+\tilde x^\top \tilde x'\sigma'(w^\top \tilde x)\sigma'(w^\top \tilde x')\right)\\
		&=\frac{1}{4 \pi}\|\tilde x\|\|\tilde x'\|\left((\pi-  \angle (\tilde x, \tilde x'))\cos\angle (\tilde x, \tilde x') +\sin \angle (\tilde x, \tilde x')\right)+\frac{1}{2 \pi}\tilde x^\top \tilde x' \left(\pi-  \angle (\tilde x, \tilde x')\right)\\
		&=\frac{1}{4 \pi}\|\tilde x\|\|\tilde x'\| U\left( \cos\angle (\tilde x, \tilde x') \right)
	\end{split}
\end{equation*}
and $U$ takes the following analytic form
\begin{equation*}
	\begin{split}
		U(t)&:= 3t(\pi-\arccos(t))+\sqrt{1-t^2}\\
		&=1+\frac{3 \pi t}{2}+\sum\limits_{i=0}^\infty \left(\frac{3(\frac{1}{2})_{i}}{(1+2i)i!}- \frac{1}{2}\frac{(\frac{1}{2})_{i}}{(i+1)!}\right)t^{2i+2}\\
		&=1+\frac{3 \pi t}{2}+\sum\limits_{i=0}^\infty\frac{(4i+5)(\frac{1}{2})_{i}t^{2i+2}}{2(2i+1)(i+1)i!}\\
	\end{split}
\end{equation*}
where $(\frac{1}{2})_i=\frac{1}{2}\times \frac{3}{2}\times \cdots\times (\frac{1}{2}+i-1)$ is the Pochhammer symbol. Now we have verified that the Neural Tangent Kernel $h^{\infty}$ is of the form \eqref{eq:cos}.

In fact, it is not difficult to prove that for multilayer fully connectedly neural network the NTK is also of this form with all positive Taylor coefficients if seen as a function of $\cos\angle (\tilde x, \tilde x')$.

\section{Proof of Theorem~\ref{thm:variance}}
\label{sec:variance_proof}

In this section we prove Theorem~\ref{thm:variance}, as it sheds light on the emergence of the multiple descent phenomenon. We need only prove (i) because (ii) shall follow easily from the proof of (i).
We will first show that with high probability, $\mathbf{K}$ is invertible, and $n\mathbf{K}\succ c \mathbf{I}_{n}$
for some constant $c > 0$.
To show this, it suffices to prove
\begin{equation*}
	n\mathbf{K}^{[\iota']}\succ  c \mathbf{I}_{n}
\end{equation*}
Write $\mathbf{K}^{[\iota']}$ as
\begin{equation*}
	n\mathbf{K}^{[\iota']}=A+B
\end{equation*}
where $A$ is the diagonal terms and $B$ is the non-diagonal terms.
With probability at least $1 - n^{-C}$, we shall have
\begin{equation}
	\| A \|_{\rm op}  = \Omega(1).
\end{equation}
and by H\"{o}lder's inequality on matrix induced norm,
\begin{align}
	&\| B\|_{\rm op} = \| B\|_{\ell_2 \rightarrow \ell_2}  \leq \sqrt{\| B\|_{\ell_1 \rightarrow \ell_1} \| B\|_{\ell_\infty \rightarrow \ell_\infty}} \\
	&= \|B\|_{\ell_1 \rightarrow \ell_1}\le O\left(n \times\left(\frac{\sqrt{\log n}}{\sqrt{d}}\right)^{\iota'}\right)\le O\left(\frac{1}{\sqrt{d}}\right) \enspace.
\end{align}
Hence
\begin{align}
	n \mathbf{K} \succeq n \mathbf{K}^{[\iota']} \succeq (\| A \|_{\rm op} - \| B\|_{\rm op}) \mathbf{I}_{n} \succeq c \mathbf{I}_{n} \enspace.
\end{align}

Now back to the proof of (i).
Recall the normalized kernel matrix $\mathbf{K} = k(\bm{X}, \bm{X})/n$.
\begin{align}
	\label{eq:put-together}
		&\mathbb{E}_{x}\|k(\bm{X}, \bm{X})^{-1}k(\bm{X},x)\|^2\\
		&\lesssim \sum_{i = 0}^{\iota}\mathbb{E}_{x}\|\mathbf{K}^{-1}\frac{1}{n}(\bm{X}x)^{i}/d^{i}\|^2+\mathbb{E}_{x}\|\mathbf{K}^{-1}\frac{1}{n}\sum_{i=\iota+1}^{\infty}(\bm{X}x)^{i}/d^{i}\|^2\\
		&\lesssim\frac{1}{n^2}\sum_{i=0}^{\iota}\mathbb{E}_{x}\|\mathbf{K}^{-1}(\bm{X}x)^i/d^i\|^2+\| (n\mathbf{K})^{-1} \|_{\rm op}^2 \cdot \mathbb{E}_{x}\|\sum_{i=\iota+1}^{\infty}(\bm{X}x)^{i}/d^{i}\|^2\\
		&\lesssim \frac{1}{n^2}\sum_{i=0}^{\iota}\mathbb{E}_{x} \left[ \| ( \mathbf{K}^{[\le i]} )^{+} \|_{\rm op}^2  \cdot \|(\bm{X}x)^{i}/d^{i}\|^2 \right]+\frac{n}{d^{\iota+1}} \label{eq:4thline}
\end{align}
Now, by Proposition~\ref{prop:key}, the above quantity is at most
\begin{align*}
		&\lesssim \frac{1}{n^2}\sum_{i=0}^{\iota}\mathbb{E}_{x}\left[ d^{2i} \cdot \|(\bm{X}x)^{i}/d^{i}\|^2 \right]+\frac{n}{d^{\iota+1}} \\
		&\lesssim \frac{1}{n^2}\sum_{i=0}^{\iota}\mathbb{E}_{x}\|(\bm{X}x)^{i}\|^2+\frac{n}{d^{\iota+1}}
		\lesssim \frac{1}{n^2}\sum_{i=0}^{\iota} n d^i +\frac{n}{d^{\iota+1}} 
		\le \frac{d^{\iota}}{n}+\frac{n}{d^{\iota+1}} \enspace.
\end{align*}
The line \eqref{eq:4thline} requires some explanation. First observe that $(\bm{X}x)^i$ lies in the span of $\mathbf{K}^{[\le i]}$. Write $
	\mathbf{K} = \mathbf{K}^{[\le i]} + \mathbf{K}^{[> i]} \succeq \mathbf{K}^{[\le i]}
$,
we know that when restricted to the column space $\Phi$ of $\mathbf{K}^{[\le i]}$ (via projection operator $\Pi_{\Phi}$), the operator $\mathbf{K}$ satisfy
\begin{align*}
	\mathbf{K}|_{\Phi} =  \mathbf{K}^{[\le i]} + \Pi_{\Phi} \mathbf{K}^{[> i]} \Pi_{\Phi} \succeq \mathbf{K}^{[\le i]}
\end{align*}
Therefore for $v = (\bm{X}x)^i$ in the span of $\Phi$, we have
\begin{align*}
	\| (\mathbf{K}|_{\Phi})^{-1} v\|^2 & \leq  \| ( \mathbf{K}^{[\le i]} )^{+} \|_{\rm op} \cdot \| (\mathbf{K}|_{\Phi})^{-1/2} v \|^2 \\
	&\leq \| ( \mathbf{K}^{[\le i]} )^{+} \|_{\rm op} \cdot v^\top (\mathbf{K}^{[\le i]} )^{+} v  \leq \| ( \mathbf{K}^{[\le i]} )^{+} \|_{\rm op}^2  \cdot \| v \|^2  \enspace.
\end{align*}
Note that in above derivation, we use that by concentration $\mathbb{E}_{x}\left[ (x^\top x_j)^{2\iota}\right]  \precsim d^{\iota}$ in high probability over $\bm{X}$.
This can be seen because conditioning on $\bm{X}$, $x^\top x_j$ is sub-Gaussian with parameter $\|x_j\|_2$, and
\begin{align*}
	\mathbb{E}_{x}\left[ (x^\top x_j)^{2\iota}\right]  \precsim \|x_j\|_2^{2\iota} \precsim d^{\iota}, \quad \text{w.p. $1 - \exp(d - \log n)$}\enspace.
\end{align*}
This concludes the proof of Theorem~\ref{thm:variance}.

\section{Proof of Theorem~\ref{thm:bias}}
\label{sec:bias_proof}
	In this section we provide a proof of Theorem~\ref{thm:bias}. Define the following ``surrogate'' function for analyzing the bias term $\tilde{f}_n(x) := \frac{1}{n} \sum_{i=1}^n k(x, x_i) \rho_*(x_i)$.
		We start with splitting 
		\begin{align*}
			&\| k(x, \bm{X}) k(\bm{X}, \bm{X})^{-1} f_*(\bm{X}) - f_*(x) \|_{\mathcal{P}_X}^2 \\
			&\precsim \left\| k(x, \bm{X}) k(\bm{X}, \bm{X})^{-1} \left[f_*(\bm{X}) - \tilde{f}_n(\bm{X})\right] \right\|_{\mathcal{P}_X}^2 + \| k(x, \bm{X}) k(\bm{X}, \bm{X})^{-1} \tilde{f}_n(\bm{X})- f_*(x) \|_{\mathcal{P}_X}^2. 
		\end{align*}
		The first term is equal to $\int \left\langle k(\bm{X}, \bm{X})^{-1} k(\bm{X}, x), f_*(\bm{X}) - \tilde{f}_n(\bm{X}) \right\rangle^2  \mathcal{P}_X (dx) $, which is at most
		\begin{align*}
			&\int \left\| k(\bm{X}, \bm{X})^{-1} k(\bm{X}, x) \right\|^2 \mathcal{P}_X (dx) \cdot \| f_*(\bm{X}) - \tilde{f}_n(\bm{X}) \|^2  \\
			&= \| f_*(\bm{X}) - \tilde{f}_n(\bm{X}) \|^2 \cdot \mathbb{E}_{x \sim \mathcal{P}_X} \left\| k(\bm{X}, \bm{X})^{-1} k(\bm{X}, x) \right\|^2
		\end{align*}
		by the Cauchy-Schwartz inequality.
		By Proposition~\ref{prop:leave-one-out} in the Appendix,
		\begin{align*}
			\mathbb{E}_{\bm{X}} \| f_*(\bm{X}) - \tilde{f}_n(\bm{X}) \|^2 \precsim 1.
		\end{align*}
		For the second term, defining a vector $\tilde{V} = [\frac{\rho_*(x_1)}{n}, \ldots,  \frac{\rho_*(x_n)}{n}] \in \mathbb{R}^n$, we have $\tilde{f}_n(\bm{X}) = k(\bm{X}, \bm{X}) \tilde{V}.$
		Then
		\begin{align*}
			\| k(x, \bm{X}) k(\bm{X}, \bm{X})^{-1} \tilde{f}_n(\bm{X})- f_*(x) \|_{\mathcal{P}_X}^2 
			&=\| k(x, \bm{X}) k(\bm{X}, \bm{X})^{-1} k(\bm{X}, \bm{X}) \tilde{V}- f_*(x) \|_{\mathcal{P}_X}^2 \\
			& = \| \tilde{f}_n(x) - f_*(x) \|_{\mathcal{P}_X}^2.
		\end{align*}
		By Proposition~\ref{prop:var} in the Appendix, it holds that $\mathbb{E}_{\bm{X}} \| f_*(x) - \tilde{f}_n(x) \|_{\mathcal{P}_X}^2 \precsim \frac{1}{n}$.

\section{Discussion}

We showed that minimum norm interpolants in RKHS, under the assumptions employed in this paper, have risk that vanishes in $n$ for a wide range of scalings $d\asymp n^\alpha$, $\alpha\in(0,1)$. Notably, the places where our upper bounds become vacuous are fractions $\alpha=1/i$ for integer $i$. The phenomenon of non-monotonicity with peaks at these locations is supported by empirical evidence, and generalizes the double-descent behavior in linear regression and other models. 

A more precise description of the risk curve is an interesting research direction. In addition, it would be interesting to understand the effect of regularization on the peaks. In terms of assumptions, the i.i.d. assumption on the coordinates can certainly be lifted, and we believe similar results hold under a rotation of vectors with independent or weakly-dependent coordinates. Some degree of independence, however, is needed to capture the scaling with $d$. 

Finally, we mention that the difficulty of analyzing min-norm interpolants is greatly reduced in the \emph{noiseless} case when $y_i=f_*(x_i)$. Indeed, in this case the variance term is zero. Moreover, one can appeal to known results on lower isometry (for instance, Lemmas 8 and 9 in \citep{Rakhlin2017}) to establish that, up to polylogarithmic factors, 
$$\mathbb{E}\norm{\algo-f_*}_{L_2(\mathcal{P}_X)}^2\lesssim \mathcal{R}_n(\mathcal{F})^2, $$
the squared Rademacher averages of $\mathcal{F}$, whenever $\algo\in\mathcal{F}$. Since $\algo$ is the min-norm interpolant, we can take $\mathcal{F}$ to be the ball in $\cH$ of radius $\norm{f_*}_{K}$, yielding a consistency result $$\mathbb{E}\norm{\algo-f_*}^2_{L_2(\mathcal{P}_X)}\lesssim \frac{\norm{f_*}^2_K}{n}.$$ This can be further tightened to an upper bound in terms of $\norm{\algo}^2_K$, in high probability. In contrast, the norm $\norm{\algo}_K$ is not easily controlled in the noisy case.

%% file: appendix.tex
\section{Proof of Main Proposition}
\label{sec:main_proof}

The proof aims to establish the \textit{restricted lower isometry} behavior for the empirical kernel $\mathbf{K}$ when restricting to the eigen-space of the population covariance operator with rank $\binom{i+d}{i}$ (sorted according to the eigenvalues). We show a lower bound for the restricted lower isometry, as multiplicatively equivalent to the population eigenvalues. The approach proceeds along the lines of \citep{koltchinskii2015bounding, mendelson2014learning}. One technical contribution is establishing the ``small-ball'' property for the polynomial basis of the kernel. 

\subsection{Preparation}

We use the notation $``r_i;"$ to represent a sequence of indices $r_1\cdots r_d$. For example, we can use this to abbreviate a monomial with order $r_1+\ldots+r_d$: for monomial $p$ (recall that $x$ is a $d$-dimensional vector),
	\begin{equation}
		p_{r_i;}(x):=p_{r_1\cdots r_d}(x)=\prod_{i=1}^d (x[i])^{r_i}
	\end{equation}
where $x[i]$ denotes the $i$-th coordinate of $x$. This notation is also used in tensors $T$, for example:
\begin{equation}
	T_{r_i;}:=T_{r_1\cdots r_d} .
\end{equation}

First, we fix an index $\iota\le \iota_0$. After we prove it for $\iota$, the conclusion shall follow easily from a union bound over $\iota=1,\cdots,\iota_0$. 

Consider the Taylor expansion, expressed with the multi-index $r_1,\cdots,r_d$ 
\begin{equation}
	\begin{split}
		n\mathbf{K}_{ij}^{[\le \iota]}&=\sum_{r_1,\cdots,r_d\ge 0}^{\iota} c_{r_1\cdots r_d}\alpha_{r_1+\cdots+r_d} p_{r_1\cdots r_d}(x_i)p_{r_1\cdots r_d}(x_j)/d^{r_1+\cdots+r_d}
	\end{split}
\end{equation}
where
\begin{equation}
	c_{r_1\cdots r_d}=\frac{(r_1+\cdots+r_d)!}{r_1!\cdots r_d!},~~ p_{r_1\cdots r_d}(x_i)=(x_{i}[1])^{r_1}\cdots (x_{i}[d])^{r_d} \enspace.
\end{equation}
Fix an ordering of all $(r_1,\cdots,r_d)$ such that $r_i\ge 0,\sum r_i\le \iota$, and let 
$$(r_1\cdots r_d)_{\iota}\in \left\{1,2,\cdots,\binom{\iota+d}{\iota}\right\}$$
denote the index of $(r_1,\cdots,r_d)$ in the ordering.
Define the $n\times \binom{\iota+d}{\iota}$ matrix $\Phi$ as
\begin{equation}
	\Phi_{i,(r_1\cdots r_d)_{\iota}}=\sqrt{c_{r_1\cdots r_d}\alpha_{r_1+\cdots+r_d}}p_{r_1\cdots r_d}(x_i)/d^{(r_1+\cdots+r_n)/2}.
\end{equation}
Then it is not hard to see
\begin{equation}
	\mathbf{K}^{[\le \iota]}=\frac{1}{n}\Phi \Phi^\top \enspace,
\end{equation}
which has the same nonzero spectrum as the covariance operator
\begin{equation}
	\Theta=\frac{1}{n}\Phi^\top \Phi.
\end{equation}
We know
\begin{equation}
	\Theta_{(r_1\cdots r_d)_{\iota},(r_1'\cdots r_d')_{\iota}}=\frac{1}{n}\sum_{i=1}^n 
	\sqrt{c_{r_1\cdots r_d}\alpha_{r_1+\cdots+r_d}}\sqrt{c_{r_1'\cdots r_d'}a_{r_1'+\cdots+r_d'}}\frac{p_{r_1\cdots r_d}(x_i)p_{r_1'\cdots r_d'}(x_i)}{d^{(r_1+\cdots+r_n+r_1'+\cdots+r_n')/2}} \enspace.
\end{equation}

It would be hard to work directly with $\Phi$ to analyze the eigenvalues of the random matrix because of complex correlation structure in the entries. Instead, we define another matrix $\Psi$ with the following properties:
\begin{enumerate}
	\item $\Psi^\top \Psi$ is easier to analyze from a probabilistic point of view;
	\item there is a linear transformation $\Lambda$ with bounded operator norm $\|\Lambda\|_{\ell_2 \rightarrow \ell_2}, \|\Lambda^{-1}\|_{\ell_2 \rightarrow \ell_2}$ such that
	\begin{equation}
		\Phi=\Psi \Lambda .
	\end{equation}
\end{enumerate}
With such a $\Lambda$, we have
	\begin{equation}
		\Theta=\Lambda^\top \Psi^\top \Psi \Lambda.
	\end{equation}
Such $\Lambda$ can be obtained through the Gram-Schimdt process on the polynomial basis, which we will elaborate on next.

\subsection{Gram-Schimdt Process} 
We now proceed in the following steps:

\paragraph{Step 1. Define $\Psi$}

	Given a distribution $\mathcal{P}$ over $\mathbb{R}$, define $q_0,q_1,\cdots,q_k,\cdots,$ to be the sequence of polynomials obtained by the Gram-Schmidt process on the basis $\{1,t,\cdots,t^k,\cdots\}$ w.r.t. the inner product of space $L^2(\mathcal{P})$. Define the $n\times \binom{\iota+d}{\iota}$ matrix $\Psi$ as
\begin{equation}
	\Psi_{i,(r_1\cdots r_d)_{\iota}}=\sqrt{c_{r_1\cdots r_d}\alpha_{r_1+\cdots+r_d}}q_{r_1\cdots r_d}(x_i)/d^{(r_1+\cdots+r_d)/2} \enspace.
\end{equation}
Here $q_{r_1\cdots r_d}$ will be defined in \eqref{eq:lemma1}.

To be concrete, one can see that
	\begin{equation}
		q_0=1,q_1=(x-m_1)/\sqrt{(m_2-m_1^2)},\cdots,
	\end{equation}
	where $m_i$ is the $i$th moment of $\mathcal{P}$,
	and for $i<j$
	\begin{equation}
		\mathbb{E}_{\mathbf{x} \sim \mathcal{P}}\left[ \mathbf{x}^i q_j(\mathbf{x}) \right]=0.
	\end{equation}

The following lemmas on properties of the polynomial basis due to Gram-Schmidt process will be useful.
\begin{lemma}
	Suppose that $f(x)=\sum_{r_i;} u_{r_i;} q_{r_i;}(x)$, then
	\begin{equation}
		\label{eq:lemma1}
		u_{r_i;}=\mathbb{E}_{X \sim \mathcal{P}^{\otimes d}} \left[ f(X) q_{r_i;}(X) \right],~\text{where}~ q_{r_i;}(x) := \prod_{i=1}^d q_{r_i}(x[i]).
	\end{equation}
\end{lemma}
\begin{lemma}
\label{lem2}
	$\mathbb{E}_{X \sim \mathcal{P}^{\otimes d}} \left[ q_{r_i;}(X)p(X) \right]=0$ if there exists one $i$ such that the degree of $p(x)$ as a polynomial of $x[i]$ is less than $r_i$. 
\end{lemma}
\begin{lemma}[Triangle condition]
	$\mathbb{E}_{X \sim \mathcal{P}^{\otimes d}} \left[ q_{r_i;}(X)q_{r_i';}(X)q_{r_i'';}(X) \right]=0$ if there exists one $i$, such that $r_i> r_i'+r_i''$. 
\end{lemma}

\paragraph{Step 2. Existence of $\Lambda$.}

\begin{lemma}
\label{piano}
	There is a $\binom{\iota+d}{\iota}\times \binom{\iota+d}{\iota}$ invertible matrix $\Lambda$ such that
	\begin{equation}
		\Phi= \Psi \Lambda.
	\end{equation}
\end{lemma}
\begin{proof}[Proof of Lemma \ref{piano}]
	For any $a\in \mathbb{R}^{\binom{\iota+d}{\iota}}$, there is a unique $b\in\mathbb{R}^{\binom{\iota+d}{\iota}}$ such that
	\begin{equation}
		\begin{split}
		\label{ninja}
			&\sum_{r_1+\cdots +r_d\le \iota,r_i\ge 0}a_{(r_1\cdots r_d)_{\iota}} \sqrt{c_{r_1\cdots r_d}\alpha_{r_1+\cdots+r_d}}p_{r_1\cdots r_d}/d^{(r_1+\cdots+r_d)/2}\\
			&=\sum_{r_1+\cdots +r_d\le \iota,r_i\ge 0}b_{(r_1\cdots r_d)_{\iota}}\sqrt{c_{r_1\cdots r_d}\alpha_{r_1+\cdots+r_d}}q_{r_1\cdots r_d}/d^{(r_1+\cdots+r_d)/2}.
		\end{split}
	\end{equation}
	As a result
	\begin{equation*}
		\Phi a=\Psi b.
	\end{equation*}
	Choose $\Lambda$ to be the linear mapping that maps $a$ to $b$
	\begin{equation*}
		\Phi a= \Psi \Lambda a.		
	\end{equation*}
	This holds for any $a$, so we have
	\begin{equation*}
		\Phi = \Psi \Lambda.
	\end{equation*}
\end{proof}

\paragraph{Step 3. Boundedness of $\Lambda$ and $\Lambda^{-1}$.}
	For a vector $v\in \mathbb{R}^{\binom{\iota+d}{\iota}}$, let $v_{\ge \iota'}$ be the vector such that
	\begin{equation}
		(v_{\ge \iota'})_{(r_1\cdots r_d)_{\iota}}=v_{(r_1\cdots r_d)_{\iota}}1_{\{r_1+\cdots+r_d\ge \iota'\}}.
	\end{equation}
	Define similarly for $v_{\iota'},v_{>\iota'}, \Lambda_{\ge \iota'}, \Lambda_{\ge \iota',<\iota'}$.
\begin{lemma}
\label{l_bdd}
	There is a constant $C(\iota)$ independent of $d$ such that
	\begin{equation}
		\|\Lambda\|_{\ell_2 \rightarrow \ell_2},\|\Lambda^{-1}\|_{\ell_2 \rightarrow \ell_2}\le C(\iota).
	\end{equation}
\end{lemma}
\begin{proof}
	[Proof of Lemma \ref{l_bdd}]
	We start with few claims about the Gram-Schmidt process and the structure of $\Lambda$.
	
	\noindent \textbf{Claim 1:} $\Lambda_{\ge \iota', <\iota'}=0$ for any $\iota'$. Alternatively, if $b=\Lambda a$ then
	\begin{equation}
		b_{\ge \iota'}=\Lambda_{\ge \iota'} a_{\ge \iota'}.
	\end{equation}
	\begin{proof}[Proof of Claim 1]

		We need only to show that if $a_{\ge \iota'}=0$, then $b_{\ge \iota'}=0$. 
		Observe that $a_{\ge \iota'}=0$ implies that the left hand of equation \eqref{ninja} is of degree less than $\iota'$. Since this is an equality, the right hand side must be of degree less than $\iota'$. Note that this implies that $b_{\ge \iota'}=0$.
	\end{proof}

	\noindent \textbf{Claim 2:} $\Lambda_{\iota',\iota'}$ is diagonal with
	\begin{equation}
		c\leq \lambda_{\min}(\Lambda_{\iota',\iota'})\ \leq \lambda_{\max}(\Lambda_{\iota',\iota'}) \leq C
	\end{equation}
	where $c,C$ depend only on $\iota', \mathcal{P}$.

	\begin{proof}[Proof of Claim 2]

	Given $r_i;$ and $r_i';$ with $\sum_ir_i=\sum_i r_i'=\iota'$, we have
	\begin{equation}
		\Lambda_{(r_i;)_{\iota},(r_i';)_{\iota}}=\frac{\sqrt{c_{r_1'\cdots r_d'}}}{\sqrt{c_{r_1\cdots r_d}}} \mathbb{E}_{X\sim \mathcal{P}^{\otimes d}}\left[ q_{r_i;}(X)p_{r_i';}(X) \right].	
	\end{equation}
	If $r_i;\neq r_i';$, there is at least one $i$ such that $r_i> r_i'$, then according to Lemma \ref{lem2}, we have
			\begin{equation}
				\Lambda_{(r_1\cdots r_d)_{\iota},(r_1'\cdots r_d')_{\iota}}=0
			\end{equation}
	Therefore $\Lambda_{\iota',\iota'}$ is diagonal. Now we have
			\begin{equation}
				\begin{split}
					\Lambda_{(r_1\cdots r_d)_{\iota},(r_1\cdots r_d)_{\iota}}&=\mathbb{E}_{X\sim \mathcal{P}^{\otimes d}} \left[ q_{r_i;}(X)p_{r_i;}(X) \right] \\
					&=\prod\limits_{i=1}^d \mathbb{E}_{\mathbf{x} \sim \mathcal{P}} [q_{r_i}(\mathbf{x})p_{r_i}(\mathbf{x})] \enspace.
				\end{split}
	\end{equation}
	Note that $q_0\equiv 1$. Since the set $I:= \{1\le i\le d: \text{$r_i$ is nonzero}\}$ is of size at most $\iota$, 
$\Lambda_{(r_i;)_{\iota},(r_i;)_{\iota}}$ is uniformly bounded by the following constant (depending on $\iota$ and $\mathcal{P}$):
		 	\begin{equation}
		 		\left(\max\limits_{ r_i} \mathbb{E}_{\mathbf{x} \sim \mathcal{P}} [q_{r_i}(\mathbf{x})p_{r_i}(\mathbf{x})] \right)^{\iota} \enspace.
		 	\end{equation}
	\end{proof}

	\noindent  \textbf{Claim 3:} Let $b_{\iota'}=\Lambda_{\iota'\iota'} a_{\iota'}+ \Lambda_{\iota',>\iota'}a_{>\iota'}$. Then $\|\Lambda_{\iota',>\iota'}\|_{\ell_2 \rightarrow \ell_2}$ has an upper bound that only depends on $\iota$.

	\begin{proof}[Proof of Claim 3]

		An entry $(\Lambda_{\iota',>\iota'})_{(r_1\cdots r_d)_{\iota},(r_1'\cdots r_d')_{\iota}}$ for $\sum_i r_i=\iota'$ can be obtained by
		\begin{align*}
		&(\Lambda_{\iota',>\iota'})_{(r_1\cdots r_d)_{\iota},(r_1'\cdots r_d')_{\iota}}\\
		&=\sqrt{\frac{c_{r_1'\cdots r_d'}\alpha_{r_1'+\cdots+r_d'}}{c_{r_1\cdots r_d}\alpha_{r_1+\cdots+r_d}}}d^{(r_1+\cdots+r_d-r_1'-\cdots-r_d')/2}\mathbb{E}_{X\sim \mathcal{P}^{\otimes d}}\left[ q_{r_i;}(X)p_{r_i';}(X) \right]\\
		&=\sqrt{\frac{c_{r_1'\cdots r_d'}\alpha_{r_1'+\cdots+r_d'}}{c_{r_1\cdots r_d}\alpha_{r_1+\cdots+r_d}}}d^{(r_1+\cdots+r_d-r_1'-\cdots-r_d')/2}
		\mathbb{E}_{\mathbf{x} \sim \mathcal{P}} [q_{r_i}(\mathbf{x})p_{r_i'}(\mathbf{x})] \enspace.
		\end{align*}
		
		Only when $\forall i, r_i\le r_i'$, the above term is nonzero and scales with $d$ in the order of $d^{(r_1+\cdots+r_d-r_1'-\cdots-r_d')/2}$. As a result,
		\begin{align*}
		b_{(r_1\cdots r_d ) _{\iota}}^2&=\left(\sum_{r_1',\cdots,r_d'}(\Lambda_{\iota',>\iota'})_{(r_1\cdots r_d)_{\iota},(r_1'\cdots r_d')_{\iota}}a_{(r_1'\cdots r_d')_{\iota}}\right)^2\\
		&=\left(\sum_{r_i\le r_i'~\forall i, ~\sum_i r_i'>\iota'}(\Lambda_{\iota',>\iota'})_{(r_1\cdots r_d)_{\iota},(r_1'\cdots r_d')_{\iota}}a_{(r_1'\cdots r_d')_{\iota}}\right)^2\\
		&\le\left(\sum_{r_i\le r_i',\sum_i r_i'>\iota'}(\Lambda_{\iota',>\iota'})_{(r_1\cdots r_d)_{\iota},(r_1'\cdots r_d')_{\iota}}^2\right)
		\left(\sum_{r_i\le r_i',\sum_i r_i'>\iota'}a_{(r_1'\cdots r_d')_{\iota}}^2\right)\\
		&\le\left(\sum_{l\ge 1}\sum_{\substack{r_i\le r_i'\\\sum_i r_i'=\iota'+l}}(\Lambda_{\iota',>\iota'})_{(r_1\cdots r_d)_{\iota},(r_1'\cdots r_d')_{\iota}}^2\right)
		\left(\sum_{r_i\le r_i',\sum_i r_i'>\iota'}a_{(r_1'\cdots r_d')_{\iota}}^2\right)\\
		&\lesssim\left(\sum_{l\ge 1}\sum_{\substack{r_i\le r_i'\\\sum_i r_i'=\iota'+l}}(d^{-l/2})^2\right)
		\left(\sum_{r_i\le r_i',\sum_i r_i'>\iota'}a_{(r_1'\cdots r_d')_{\iota}}^2\right)\\
		&\lesssim\left(\sum_{l\ge 1}d^l\times(d^{-l/2})^2\right)
		\left(\sum_{r_i\le r_i',\sum_i r_i'>\iota'}a_{(r_1'\cdots r_d')_{\iota}}^2\right)\\
		&\lesssim 
		\sum_{r_i\le r_i',\sum_i r_i'>\iota'}a_{(r_1'\cdots r_d')_{\iota}}^2 \enspace.
		\end{align*}
		Then we have
		\begin{align*}
		&\sum_{\sum_i r_i=\iota'}b_{(r_1\cdots r_d ) _{\iota}}^2\\\
		&\lesssim 
		\sum_{\sum_i r_i=\iota'	}
		\sum_{r_i\le r_i',\sum_i r_i'>\iota'}a_{(r_1'\cdots r_d')_{\iota}}^2\\
		&\lesssim \sum_{\sum_i r_i'>\iota'} a^2_{(r_1'\cdots r_d')_{\iota}}
		\end{align*}
		where the last inequality holds because for a fixed $r_1',\cdots,r_d'$ with $r_1'+\cdots + r_d'=\iota$, there is at most $2^{\iota}$ choices of $r_1,\cdots,r_d$ such that $\forall i, 0\le r_i\le r_i'$.
	\end{proof}
	
	Using induction on $\iota'\le \iota$ backwards will complete the proof that $\| b \| \asymp \| a \|$. Specifically,
	assume that $\|b_{>\iota'}\| \asymp \|a_{>\iota'}\|$, then since
	\begin{equation}
		\|b_{\iota'}\|\le \|\Lambda_{\iota'\iota'} a_{\iota'}\|+ \|\Lambda_{\iota',>\iota'}a_{>\iota'}\|\lesssim \|a_{\ge \iota'}\|,
	\end{equation}
	we have
	\begin{equation}
		\|b_{\ge \iota'}\|\le \|b_{\iota'}\|+\|b_{>\iota'}\|\lesssim \|a_{\ge \iota'}\|.
	\end{equation}
	For the other direction, we have
	\begin{equation*}
		\|a_{\iota'}\|=\|\Lambda_{\iota'\iota'}^{-1}(b_{\iota'}- \Lambda_{\iota',>\iota'}a_{>\iota'})\|\lesssim \|b_{\iota'}- \Lambda_{\iota',>\iota'}a_{>\iota'}\|\lesssim \|b_{\iota'}\|+\|a_{>\iota'}\|\lesssim  \|b_{\iota'}\|+\|b_{>\iota'}\|\lesssim\|b_{\ge \iota'}\|.
	\end{equation*}
	Then
	\begin{equation}
		\|a_{\ge \iota'}\|\lesssim \|b_{\ge \iota'}\|.
	\end{equation}
\end{proof}

Now we can proceed to analyze the spectrum of $\Psi^\top \Psi$, given the boundedness of $\Lambda$.

\subsection{Small Ball Property}

	Define the following function over $\mathbb{R}^d$ indexed by $u \in \mathbb{R}^{\binom{\iota+d}{\iota}}$
	\begin{equation}
		f_u(x)=\sum_{r_1,\cdots,r_d\ge 0, \sum_i r_i\le \iota}u_{(r_1\cdots r_d)_{\iota}}\sqrt{c_{r_1\cdots r_d} \alpha_{r_1+\cdots+r_d}}q_{r_1\cdots r_d}(x)/d^{(r_1+\cdots+r_d)/2} \enspace.
	\end{equation}
	In this section, we will prove that there exist constants $0 < \epsilon, \delta<1$, such that for any $u$ with $\norm{u}=1$,
			 \begin{equation}
				 \label{eq:small-ball}
			 	\mathbb{P}(f_u(X)^2> \epsilon \mathbb{E}[f_u(X)^2])\ge \delta \enspace.
			 \end{equation}
	This is so called the \textit{small-ball property} for the random variable $f_u(X)$, with $X\sim \mathcal{P}^{\otimes d}$.

\paragraph{Claim 1:}
		$\forall u\in \mathbb{R}^{\binom{\iota+d}{\iota}}$ with $\|u\|_2=1$, there are constants $\delta, \epsilon>0$ depending only on $\iota, \mathcal{P}$ such that with probability at least $1- \delta$
		\begin{equation}
		 	f_u(x_i)^2> \epsilon d^{-\iota}.
		 \end{equation} 

\noindent \textit{Proof of Claim 1.} First, according to the Paley-Zygmund inequality for $X\sim \mathcal{P}^{\otimes d}$ and any $0\le \theta\le 1$,
		 \begin{equation}
		 	\mathbb{P}(f_u(X)^2> \theta \mathbb{E}[f_u(X)^2])\ge (1- \theta)^2 \frac{\mathbb{E}[f_u(X)^2]^2}{\mathbb{E}[f_u(X)^4]} \enspace.
		 \end{equation}
		 Therefore we just need to show that
		 \begin{equation}
		 \label{husky}
		 	\mathbb{E}[f_u(X)^2]\gtrsim d^{-\iota},
		 \end{equation}
		 and
		 \begin{equation}
		 \label{alaskan_malamute}
		 	\mathbb{E}[f_u(X)^4]\lesssim \mathbb{E}[f_u(X)^2]^2.
		 \end{equation}

For equation \eqref{husky}, we have by the orthogonality of $q_{r_i;}$,
		 \begin{equation}
		 	\begin{split}
		 		\mathbb{E}[f_u(X)^2]&= \sum_{r_1,\cdots,r_d\ge 0, \sum_i r_i\le \iota}u^2_{(r_1\cdots r_d)_{\iota}}c_{r_1\cdots r_d} \alpha_{r_1+\cdots+r_d}/d^{r_1+\cdots+r_d}\\
		 		&\gtrsim\sum_{r_1,\cdots,r_d\ge 0, \sum_i r_i\le \iota}u^2_{(r_1\cdots r_d)_{\iota}}/d^{r_1+\cdots+r_d} \gtrsim d^{-\iota} \enspace.
		 	\end{split}
		 \end{equation}

		 Equation \eqref{alaskan_malamute} is more technical to establish, which we prove through the following lemma.
		 \begin{lemma}
			 \label{lem:l2l4}
		 	Let $f_ \gamma(x):= \sum\limits_{r_1,\cdots,r_d\ge 0, \sum_i r_i\le \iota}\gamma_{r_1\cdots r_d}q_{r_1\cdots r_d}(x)$, then
		 	\begin{equation}
		 		\mathbb{E}f_\gamma(X)^4\lesssim (\mathbb{E}f_ \gamma(X))^2)^2=\left(\sum_{\sum_i r_i\le \iota}\gamma_{r_1\cdots r_d}^2\right)^2 \enspace.
		 	\end{equation}
		 \end{lemma}

		 \begin{proof}
		 	[Proof of Lemma~\ref{lem:l2l4}. ] Write
		 	\begin{equation}
		 		f_ \gamma(x)^2=\sum_{\sum_i r_i\le 2 \iota} \theta_{r_1\cdots r_d} q_{r_1\cdots r_d}(x) \enspace.
		 	\end{equation}
		 	Since
		 	\begin{equation}
		 		f_ \gamma(x)^2=\sum_{r_i'; r_i'';} \gamma_{r_i';} \gamma_{r_i'';}q_{r_i';}(x)q_{r_i'';}(x) \enspace,
		 	\end{equation}
		 	then by triangle condition there are coefficients $T_{r'_i;r''_i;}^{r_i;}, M^{r_i;}_{r'_i;r_i'';}$ (defined in \eqref{eq:T} and \eqref{eq:M}) such that
		 	\begin{equation*}
		 		\begin{split}
		 			\theta_{r_1\cdots r_d}&=\mathbb{E}_{X\sim \mathcal{P}^{\otimes d}} \left[ q_{r_i;}(X)f_ \gamma(X)^2 \right]\\
		 			&=\mathbb{E}_{X\sim \mathcal{P}^{\otimes d}} \left\{ q_{r_i;}(X)\sum_{r_i';r_i'';} \gamma_{r_i';} \gamma_{r_i'';}q_{r_i';}(X)q_{r_i'';}(X) \right\}\\
		 			&=\sum_{r_i';r_i'';} \gamma_{r_i';} \gamma_{r_i'';} \mathbb{E}_{X\sim \mathcal{P}^{\otimes d}} \left[q_{r_i;}(X)q_{r_i';}(X)q_{r_i'';}(X) \right]\\
		 			&=\sum_{\forall i, r_i'+r_i''\ge r_i} T^{r_i;}_{r_i';r_i'';}\gamma_{r_i';} \gamma_{r_i'';} \\
		 			&=\sum_{s_i'+s_i''= r_i}\sum_{\substack{s_i'\le r_i',s_i''\le r_i''\\ T^{r_i}_{r_i'; r_i'';}\neq 0}} T^{r_i;}_{r_i';r_i'';}\gamma_{r_i';} \gamma_{r_i'';}/M^{r_i;}_{r_i';r_i'';}\\
		 		\end{split}
		 	\end{equation*}
		 	where the coefficients are given by
		 	\begin{equation}
				\label{eq:T}
		 		T_{r'_i;r''_i;}^{r_i;}:=\mathbb{E}_{X\sim \mathcal{P}^{\otimes d}} \left[q_{r_i;}(X)q_{r_i';}(X)q_{r_i'';}(X) \right],
		 	\end{equation}
		 	and
		 	\begin{equation}
				\label{eq:M}
		 		M_{r'_i;r''_i;}^{r_i;}:=\# \{(s_i';,s_i'';):s_i'+s_i''=r_i,s_i'\le r_i',s_i''\le r_i''\}.
		 	\end{equation}

		 	Now we will upper bound $T^{r_i;}_{r_i';r_i'';}$
		 	\begin{align*}
		 	&|T_{r'_i;r''_i;}^{r_i;}|\\
		 	&=\left|\mathbb{E}_{X\sim \mathcal{P}^{\otimes d}} \left[q_{r_i;}(X)q_{r_i';}(X)q_{r_i'';}(X) \right]\right|\\
		 	&=\left|\mathbb{E}_{X\sim \mathcal{P}^{\otimes d}} \prod_{i=1}^dq_{r_i}(X[i])q_{r_i'}(X[i])q_{r_i''}(X[i]) \right|\\
		 	&=\left|\prod_{i=1}^d \mathbb{E}_{\mathbf{x}\sim \mathcal{P}}q_{r_i}(\mathbf{x})q_{r_i'}(\mathbf{x})q_{r_i''}(\mathbf{x})\right| \enspace.
		 	\end{align*}
			Note that $q_0\equiv 1$. Since the set $I:= \{i \in [d]: \text{one of $r_i,r_i',r_i''$ is nonzero}\}$ is of size at most $3\iota$, 
$T^{r_i;}_{r'_i;r''_i;}$ is uniformly bounded by the following constant (depending on $\iota$ and $\mathcal{P}$) for $\sum r_i,\sum r_i',\sum r_i''\le \iota$:
		 	\begin{equation}
		 		\left(\max\limits_{ r_i,r'_i,r''_i} \mathbb{E}_{X\sim \mathcal{P}} \left[ q_{r_i}(X)q_{r_i'}(X)q_{r_i''}(X) \right] \right)^{3\iota} \enspace.
		 	\end{equation}

		 	As a result, we have
		 	\begin{equation}
		 		\begin{split}
		 			&\theta^2_{r_1\cdots r_d}\\
		 			=&\left(\sum_{s_i'+s_i''= r_i}\sum_{s_i'\le r_i',s_i''\le r_i''} T^{r_i;}_{r_i';r_i'';}\gamma_{r_i';} \gamma_{r_i'';}/M^{r_i;}_{r_i';r_i'';}\right)^2\\
		 			\lesssim&\sum_{s_i'+s_i''=r_i}\left(\sum_{s_i'\le r_i',s_i''\le r_i''} T^{r_i;}_{r_i';r_i'';}\gamma_{r_i';} \gamma_{r_i'';}/M^{r_i;}_{r_i';r_i'';}\right)^2 \quad \text{ignoring constants depending only on $\iota$}\\
		 			\lesssim&\sum_{s_i'+s_i''=r_i}
		 			\left(\sum_{\substack{s_i'\le r_i',s_i''\le r_i''\\T^{r_i;}_{r_i';r_i'';}\neq 0}} (T^{r_i;}_{r_i';r_i'';}/M^{r_i;}_{r_i';r_i'';})^2\gamma_{r_i';}^2\right)
		 			\left(\sum_{\substack{s_i'\le r_i',s_i''\le r_i''\\T^{r_i;}_{r_i';r_i'';}\neq 0}} \gamma_{r_i'';}^2\right) \quad \text{Cauchy inequality}
		 			\\
		 			\lesssim&\sum_{s_i'+s_i''=r_i}
		 			\left(\sum_{\substack{s_i'\le r_i',s_i''\le r_i''\\T^{r_i;}_{r_i';r_i'';}\neq 0}} \gamma_{r_i';}^2\right)
		 			\left(\sum_{\substack{s_i'\le r_i',s_i''\le r_i''\\T^{r_i;}_{r_i';r_i'';}\neq 0}} \gamma_{r_i'';}^2\right) \enspace.
		 		\end{split}
		 	\end{equation}

		 	Note that for $T^{r_i;}_{r_i';r_i'';}\neq 0$, we must have by the triangle condition
		 	\begin{equation}
		 		\forall i, r_i''\le r_i'+r_i
		 	\end{equation}
		 	which means that for any $r_i''\neq 0$, either $r_i'\neq 0$ or $r_i\neq 0$. Then
		 	\begin{equation}
		 		\{i\in[d]:r_i''\neq 0\}\subset \{i\in[d]:r_i'\neq 0\}\cup\{i\in[d]:r_i\neq 0\} .
		 	\end{equation}

		 	As a result, for fixed $r_i';$ and $r_i;$~, there is less than constantly many $r_i'';$ such that $T^{r_i;}_{r_i';r_i'';}\neq 0$. Similarly, for fixed $r_i'';$ and $r_i;$, there is less than constantly many $r_i';$ such that $T^{r_i;}_{r_i';r_i'';}\neq 0$.

		 	We now have
		 	\begin{equation}
		 		\begin{split}
		 			\theta^2_{r_1\cdots r_d}&\lesssim\sum_{s_i'+s_i''=r_i}
		 			\left(\sum_{\substack{s_i'\le r_i',s_i''\le r_i''\\T^{r_i;}_{r_i';r_i'';}\neq 0}} \gamma_{r_i';}^2\right)
		 			\left(\sum_{\substack{s_i'\le r_i',s_i''\le r_i''\\T^{r_i;}_{r_i';r_i'';}\neq 0}} \gamma_{r_i'';}^2\right)
		 			\\
		 			&\lesssim\sum_{s_i'+s_i''=r_i}
		 			\left(\sum_{s_i'\le r_i'}\gamma_{r_i';}^2\right)
		 			\left(\sum_{s_i''\le r_i''} \gamma_{r_i'';}^2\right)
		 			\\
		 			&=\sum_{s_i'+s_i''=r_i}
		 			\sum_{s_i'\le r_i'}\sum_{s_i''\le r_i''}
		 			\gamma_{r_i';}^2
		 			\gamma_{r_i'';}^2
		 			\\
		 		\end{split}
		 	\end{equation}

		 	As a result,
		 	\begin{equation}
		 		\sum_{r_i}\theta^2_{r_1\cdots r_d}\lesssim\sum_{r_i}\sum_{s_i'+s_i''=r_i}
		 			\sum_{s_i'\le r_i'}\sum_{s_i''\le r_i''}
		 			\gamma_{r_i';}^2
		 			\gamma_{r_i'';}^2
		 	\end{equation}
		 	Note that on the RHS, for a fixed $r'_i;, r''_i;$ the term $\gamma_{r'_i;}^2 \gamma_{r''_i;}^2$ appears constantly many times (with constant relying on $\iota$ only), since $M_{r'_i;r''_i;}^{r_i;} \leq 2^{2\iota}$ and there are at most $2^{2\iota}$ number of $r_i;$'s such that $M_{r'_i;r''_i;}^{r_i;} > 0$.
			Finally we have
		 	\begin{equation}
		 		\mathbb{E}[f_ \gamma(X)^4]=\sum_{r_i}\theta^2_{r_1\cdots r_d}\lesssim\sum_{r_i',r_i''} \gamma_{r_i';}^2\gamma_{r_i'';}^2= (\mathbb{E} [f_ \gamma(X)^2])^2 \enspace.
		 	\end{equation}
			\end{proof}

\subsection{Lower Isometry}

We now proceed to lower bound the smallest eigenvalue for $\frac{1}{n}\Psi^\top \Psi$, based on the small-ball property established.

\begin{lemma}
\label{mhw}
	With probability at least $1-e^{- \Omega(n/d^{\iota})}$, the smallest eigenvalue of $\frac{1}{n} \Psi^\top \Psi$ is larger than $Cd^{-\iota}$.
\end{lemma}

\begin{enumerate}[Step 1.]	 
	\item We will first prove: there is $\epsilon>0$ such that for any  $u\in \mathbb{R}^{\binom{\iota+d}{\iota}}$ with $\|u\|_2=1$,
	\begin{equation}
		\mathbb{P}\left(u^\top \Psi^\top \Psi u\ge c n d^{-\iota} \right)\ge 1- e^{-c' n}.
	\end{equation}

	Since
	\begin{equation}
		 u^\top \Psi^\top \Psi u= \sum\limits_{i=1}^n f_u(x_i)^2
	\end{equation}
	with $f_u(x_i)^2\ge 0$ i.i.d. drawn.
	Define $Z_i=1_ {\{f_u(x_i)^2\ge \epsilon d^{-\iota}\}}$. According to the Claim 1 (the small ball property \eqref{eq:small-ball}), we can choose $\epsilon$ such that $\mathbb{E}Z_i> C(\iota,\mathcal{P})>0$. Denote this $C(\iota,\mathcal{P})$ as $\delta$. Now we have
	\begin{equation}
		u^\top \Psi^\top \Psi u\ge \epsilon d^{-\iota} \sum\limits_{i=1}^n Z_i .
	\end{equation}
	Using the Hoeffding's inequality,
	\begin{equation*}
		\mathbb{P}\left(\sum_{i=1}^n Z_i\le n(\delta - t) \right)\le e^{-2t^2n} .
	\end{equation*}
	Take $t=\delta/2$, we get that
	\begin{equation*}
		\mathbb{P}\left(u^\top ( \frac{1}{n} \Psi^\top \Psi ) u\ge \frac{\epsilon \delta }{2 d^{\iota}} \right)\ge 1- e^{-\delta^2 n/2}.
	\end{equation*}
	\item Now we show that there is constant $C$ such that for any $L>0$, $ u^\top (1/n \cdot \Psi^\top \Psi )u$ is $L$-Lipschitz w.r.t. $u$ on the sphere, with probability at least $1 - C\left( \frac{n^2}{d} L^{-\frac{2}{\iota}} + n^{1-\frac{1}{2\iota}} L^{-\frac{1}{\iota}} \right)$.
	In fact, for $\|u\|=\|v\|=1$, we have
	\begin{align*}
			&\|u^\top \Psi^\top \Psi u-v^\top \Psi^\top \Psi v\|\\
			&=\|u^\top \Psi^\top \Psi u-v^\top \Psi^\top \Psi u+u^\top \Psi^\top \Psi v-v^\top \Psi^\top \Psi v\|\\
			&=\|(u-v)^\top \Psi^\top \Psi u\|+\|(u-v)^\top \Psi^\top \Psi v\|\\
			&\le\|u-v\|\|\Psi^\top \Psi\|_{\ell_2\rightarrow \ell_2}\|u\|+\|u-v\|\|\Psi^\top \Psi\|_{\ell_2\rightarrow \ell_2} \| v\|\\
			&=2\|u-v\|\|\Psi^\top \Psi\|_{\ell_2\rightarrow \ell_2}.
	\end{align*}
	Therefore, the map $u\to u^\top (1/n \cdot \Psi^\top \Psi) u$ is $2\|1/n \cdot \Psi^\top \Psi\|_{\ell_2\rightarrow \ell_2}$-Lipschitz.
Now we need to bound the spectral norm of $\Psi^\top \Psi$. We have
	\begin{align*}
					\|1/n \cdot \Psi^\top \Psi\|_{\ell_2\rightarrow \ell_2}^2
					&=\|1/n \cdot \Lambda^{-1\top}\Phi^\top \Phi \Lambda^{-1}\|_{\ell_2\rightarrow \ell_2}^2\\
					&\lesssim\|1/n \cdot \Phi^\top \Phi \|_{\ell_2\rightarrow \ell_2}^2 =\| 1/n \cdot \Phi \Phi^\top\|_{\ell_2\rightarrow \ell_2}^2\\
					&=\| \mathbf{K}^{[\le \iota]}\|_{\ell_2\rightarrow \ell_2}^2
					= \frac{1}{n^2}\sum_{1\le ij\le n} \left(h^{[\le \iota]}( x_i^\top x_j/d)\right)^2
	\end{align*}
	The last quantity is at most
	\begin{align*}
					&\frac{1}{n^2}\sum_{1\le i,j\le n} \left(\sum_{i=0}^{\iota}\alpha_i|x_i^\top x_j/d|^{i}\right)^2 \lesssim \frac{1}{n^2}\sum_{i,j} \left( 1+(x_{i}^\top x_j/d)^{2\iota} \right)\\
					& = \frac{1}{n^2} \sum_{ij} 1 + Y_{ij}^{2\iota}
	\end{align*}	
	with $Y_{ij}:= \frac{x_i^\top x_j}{d} $. We know that by Chebyshev's and Markov's inequality, for any $i \neq j\in [n]$, due to $\mathbb{E} \left[x_i[k]\right] = 0$,
	\begin{align}
		\mathbb{P}\left(\frac{1}{d}\sum_{k=1}^d x_i[k] x_j[k] \geq t \right) \leq \frac{C}{d t^2}, ~~\text{and} \\
		\mathbb{P}\left(\frac{1}{d}\sum_{k=1}^d (x_i[k])^2 \geq s \right) \leq \frac{C}{s} \enspace.
	\end{align}
	Choose $t \asymp L^{\frac{1}{\iota}}$, and $s \asymp n^{\frac{1}{2\iota}} L^{\frac{1}{\iota}}$
	Therefore, with probability $$1 - \frac{C n^2}{d t^2}-  \frac{Cn}{s} = 1 - C\left( \frac{n^2}{d} L^{-\frac{2}{\iota}} + n^{1-\frac{1}{2\iota}} L^{-\frac{1}{\iota}} \right),$$
	the following bound on the Lipchitz constant holds 
	\begin{align}
		\| 1/n \cdot \Psi^\top \Psi\|_{\ell_2\rightarrow \ell_2} \lesssim 2(1+L^2)^{1/2} \asymp L.
	\end{align}

	\item Suppose $\frac{d^{\iota}\log d}{n}=o(1)$, we will show that with probability at least $1-e^{- \Omega(n/d^{\iota})}$,
	\begin{equation}
		\inf_{ \| u\| = 1} u^\top \left(\frac{1}{n}\Psi^\top \Psi \right) u = \Omega(d^{-\iota})\enspace.
	\end{equation}
	
	Make an $r$-covering net $\mathcal{N}_{r}$ of the unit sphere $\| u\| = 1$ with radius 
	\begin{align}
		r = \frac{1}{2L} \frac{\epsilon \delta}{8} d^{-\iota}.
	\end{align}
	Clearly, the cardinality of such $r$-covering is bounded by $|\mathcal{N}_{r}| \leq (1+1/\epsilon)^{\binom{\iota+d}{\iota}}$. Therefore, with probability at least $1 - (1+1/r)^{\binom{\iota+d}{\iota}} e^{- \delta^2n/2}$, we know for any elements $v \in \mathcal{N}_{r}$ in the $\epsilon$-cover,
	\begin{align}
		v^\top \left(\frac{1}{n}\Psi^\top \Psi \right) v \geq \frac{\epsilon \delta }{2 d^{\iota}}.
	\end{align}
	
	Recall that with probability at least $1 - C nd L^{-\frac{\nu}{2\iota}}$, $u^\top (1/n \cdot \Psi^\top \Psi )u$ is $L$-Lipschitz w.r.t. $u$ on the sphere. For any $u$, there exists $v \in \mathcal{N}_{r}$, such that
	\begin{align*}
		 u^\top \left(\frac{1}{n}\Psi^\top \Psi \right) u -  v^\top \left(\frac{1}{n}\Psi^\top \Psi \right) v \geq - L \|u - v \| \geq -L r = -\frac{\epsilon \delta}{16} d^{-\iota}\\
		 u^\top \left(\frac{1}{n}\Psi^\top \Psi \right) u \geq v^\top \left(\frac{1}{n}\Psi^\top \Psi \right) v -\frac{\epsilon \delta}{16} d^{-\iota} \geq \frac{\epsilon \delta}{4} d^{-\iota} \enspace.
	\end{align*}
	So far we have proved with that probability at least
	\begin{align}
		1 - C\left( \frac{n^2}{d} L^{-\frac{2}{\iota}} + n^{1-\frac{1}{2\iota}} L^{-\frac{1}{\iota}} \right) - (1+1/r)^{\binom{\iota+d}{\iota}} e^{-\frac{\delta^2}{2} n}
	\end{align}
	the following holds
	\begin{equation}
		\inf_{ \| u\| = 1} u^\top \left(\frac{1}{n}\Psi^\top \Psi \right) u = \Omega(d^{-\iota})\enspace.
	\end{equation}
	Now let's control the probability via a proper choice of $L$ and $r$.
	If we choose $L = \exp(\iota(n-d^\iota \log d)) n^{\iota-\frac{1}{2}}$, then it is easy to verify that $\frac{n^2}{d} L^{-\frac{2}{\iota}} + n^{1-\frac{1}{2\iota}} L^{-\frac{1}{\iota}}  \leq \exp(-c' (n - d^\iota \log d))$, and
	\begin{align}
		(1+3/r)^{\binom{\iota+d}{\iota}} e^{-\frac{\delta^2}{2} n} \leq \exp( d^{\iota} \log(d^{\iota} L/c')  - cn ) \leq \exp(- c'' n/d^{\iota}).
	\end{align}

	\item Put things together, we have shown that w.h.p.,
	\begin{equation}
		\Theta= \Phi^\top \Phi/n=\Lambda^\top \Psi^\top \Psi \Lambda/n\succ c d^{-\iota} .
	\end{equation}

	As a result $\mathbf{K}^{[\iota]}=\Phi \Phi^\top/n$ has $\binom{\iota+d}{\iota}$ number of nonzero eigenvalues, all of them larger than $c d^{-\iota}$.
\end{enumerate}


\section{Technical Proofs}
\subsection{Proofs in Section~\ref{sec:RKHS}}

\begin{prop}[Leave-one-out]
	\label{prop:leave-one-out}
	\begin{align}
		\mathbb{E}_{\bm{X}} \| f_*(\bm{X}) - \tilde{f}_n(\bm{X}) \|^2 \precsim 1
	\end{align}
\end{prop}
\begin{proof}[Proof of Proposition~\ref{prop:leave-one-out}]
	
	We claim that,
	\begin{align*}
		& | f_*(x_j) - \tilde{f}_n(x_j) |^2 \\
		&\leq  \frac{1}{n^2} \left|  f_*(x_j) -   k(x_j, x_j) \rho_*(x_j) \right|^2 + \frac{(n-1)^2}{n^2} \left| \frac{1}{n-1}\sum_{i\neq j} k(x_j, x_i) \rho_*(x_i) - f_*(x_j) \right|^2 
	\end{align*}
	We know that for any $x, z \in \mathcal{\bm{X}}$
$k(z, x) \leq \sqrt{k(z, z) k(x, x)} \leq '$ .
	Therefore we have 
	\begin{align*}
		\mathbb{E}_{\bm{X}} \frac{1}{n^2} \left| f_*(x_j) - k(x_j, x_j) \rho_*(x_j) \right|^2 \leq \frac{1}{n^2},
	\end{align*}
	since $\sup_{x \in \mathcal{X}}K(x, x) \leq C$, $\| \rho_* \|_{\mathcal{P}_X} \leq C$, and by Jensen's inequality
	\begin{align*}
		\int f_*(x)^2 \mathcal{P}_X(dx) \leq \int \int k(z, x)^2 \rho_\star(z)^2 \mathcal{P}_X (dz) \mathcal{P}_X(dx) \leq  C^2.
	\end{align*}
	For the leave-on-out term, 
	\begin{align*}
		&\mathbb{E}_{\bm{X}} \frac{(n-1)^2}{n^2}| \frac{1}{n-1}\sum_{i\neq j} k(x_j, x_i) \rho_*(x_i) - f_*(x_j) |^2 \\
		& \precsim \frac{(n-1)^2}{n^2} \mathbb{E}_{x_j}\left[  \mathbb{E}_{\bm{X} \backslash x_j} [ | \frac{1}{n-1}\sum_{i\neq j} k(x_j, x_i) \rho_*(x_i) - f_*(x_j) |^2 |x_j ] \right] \\
		& \precsim \frac{(n-1)^2}{n^2}\mathbb{E}_{x_j}\left[ \frac{1}{n-1} \int k^2(x_j, x) \rho_*^2(x) \mathcal{P}_X(dx)  \right] \precsim \frac{1}{n}.
	\end{align*}
	Therefore we have
	\begin{align*}
		\mathbb{E}_{\bm{X}} \| f_*(\bm{X}) - \tilde{f}_n(\bm{X}) \|^2 = \mathbb{E}_{\bm{X}} \sum_{j=1}^n | f_*(x_j) - \tilde{f}_n(x_j) |^2 \leq n\left( \frac{1}{n^2} + \frac{1}{n}\right) \precsim 1.
	\end{align*}	
\end{proof}

\begin{prop}[Variance]
	\label{prop:var}
	
	\begin{align}
		n \mathbb{E}_{\bm{X}} \| f_*(x) - \tilde{f}_n(x) \|_{\mathcal{P}_X}^2 \precsim 1
	\end{align}
\end{prop}
\begin{proof}[Proof of Proposition~\ref{prop:var}]
	\begin{align}
		\mathbb{E}_{\bm{X}} \int \left( f_*(x) - \tilde{f}_n(x) \right)^2 \mathcal{P}_X (dx) \leq  \frac{1}{n} \int \int k^2(x, x') \rho^2_*(x') \mathcal{P}_X(dx')\mathcal{P}_X(dx) \leq C,
	\end{align}
	by boundedness of $ k^2(x, x') \leq C^2$ and $\| \rho_\star\|_{\mathcal{P}_X} \leq C$.
\end{proof}

\begin{prop}[Probability bound]
	\label{prop:in-prob}
	The following bounds hold simultaneously with probability at least $1-\delta$ on $\bm{X}$,
	\begin{align*}
		C_1(\bm{X})&:=  n\| f_*(x) - \tilde{f}_n(x) \|_{\mathcal{P}_X}^2  \precsim \frac{1}{\sqrt{\delta}}, \\
		C_2(\bm{X})&:= \| f_*(\bm{X}) - \tilde{f}_n(\bm{X}) \|^2 = \sum_{j=1}^n  \left( f_*(x_j) - \tilde{f}_n(x_j) \right)^2 \precsim \frac{1}{\sqrt{\delta}}
	\end{align*}
\end{prop}
\begin{proof}[Proof of Proposition~\ref{prop:in-prob}]
	We have shown that $\mathbb{E}_{\bm{X}} C_1(\bm{X}) \precsim 1/n$, and $\mathbb{E}_{\bm{X}} C_2(\bm{X}) \precsim 1$. Let's use the second moment method to show the in probability bounds for both terms. Define $\tilde{h}(x, x_i):= k(x, x_i) \rho_*(x_i) - f_*(x)$.
It is clear that $\mathbb{E}_{x_i \sim \mathcal{P}_X}[\tilde{h}(x, x_i)] = 0$ for any fixed $x$.
	\paragraph{Second moment calculations on $C_1(\bm{X})$.}
	\begin{align*}
		&\mathbb{E} \left[ C_1(\bm{X})^2 \right] = n^2 \mathbb{E} \left[ \frac{1}{n^2}\sum_{i, j} \int (k(x, x_i) \rho_*(x_i) - f_*(x) )(k(x, x_j) \rho_*(x_j) - f_*(x) ) \mathcal{P}_X(dx)  \right]^2 \\
		& = \frac{1}{n^2} \sum_{i,j,k,l} \mathbb{E}  \left[ \int \tilde{h}(x, x_i)\tilde{h}(x, x_j) \mathcal{P}_X(dx) \right] \left[\int \tilde{h}(x, x_k)\tilde{h}(x, x_l)  \mathcal{P}_X(dx) \right]
	\end{align*}
	Clearly the only nonzero terms on the RHS must be either (1) $(i,j) = (k, l)$, or (2) $(i, j) \neq (k, l)$ but $i=j$ and $k=l$. In case (1), we know
	\begin{align*}
		\sum_{i,j} \mathbb{E}  \left[ \int \tilde{h}(x, x_i)\tilde{h}(x, x_j) \mathcal{P}_X(dx)  \right]^2 \precsim n^2.
	\end{align*}
	In case (2), we know
	\begin{align*}
		\sum_{i \neq k} \mathbb{E}  \left[ \int \tilde{h}(x, x_i)^2 \mathcal{P}_X(dx) \right] \left[\int \tilde{h}(x, x_k)^2  \mathcal{P}_X(dx)  \right] \precsim n^2.
	\end{align*}
	All other terms, must have form $\mathbb{E}  \left[ \int \tilde{h}(x, x_i) \tilde{h}(x, x_j) \mathcal{P}_X(dx)  \right] \left[\int \tilde{h}(x, x_k)^2  \mathcal{P}_X(dx) \right] = 0$ for $i\neq j$, or $\mathbb{E} \left[ \int \tilde{h}(x, x_i)\tilde{h}(x, x_j) \mathcal{P}_X(dx)  \right] \left[\int \tilde{h}(x, x_k)\tilde{h}(x, x_l)  \mathcal{P}_X(dx)  \right] = 0$ for $i\neq j, k \neq l, (i,j)\neq (k,l)$. 
	Therefore, we have the second moment bound $\mathbb{E} \left[ C_1(\bm{X})^2 \right] \precsim 1$, by Chebyshev's inequality, we have the desired bound. 
	
	\paragraph{Second moment calculations on $C_2(\bm{X})$.}
	\begin{align*}
		&\mathbb{E} \left[ C_2(\bm{X})^2 \right] = \sum_{i,j} \mathbb{E} \left[ \left(f_*(x_i) - \tilde{f}_n(x_i)\right)^2 \left(f_*(x_j) - \tilde{f}_n(x_j)\right)^2 \right]\\
		& \leq \sum_{i,j} \left[ \mathbb{E}\left(f_*(x_i) - \tilde{f}_n(x_i)\right)^4 \right]^{1/2} \left[ \mathbb{E}\left(f_*(x_j) - \tilde{f}_n(x_j)\right)^4 \right]^{1/2}
	\end{align*}
	We know that
	\begin{align*}
		\mathbb{E}\left(f_*(x_i) - \tilde{f}_n(x_i)\right)^4 & = \frac{1}{n^4} \sum_{j_1, j_2, j_3, j_4} \mathbb{E}[\tilde{h}(x_i, x_{j_1})\tilde{h}(x_i, x_{j_2})\tilde{h}(x_i, x_{j_3})\tilde{h}(x_i, x_{j_4})].
	\end{align*}
	Divide into two cases: (1) some $j$ equals $i$, (2) all $j$'s do not equal $i$. In the first case, the only nonzero terms are, of the form
	$\tilde{h}(x_i, x_i)^2 \tilde{h}(x_i, x_{j})^2$ with $j \neq i$, or of the form $\tilde{h}(x_i, x_{i})\tilde{h}(x_i, x_{j})^3$ with $j \neq i$. In both cases there are at most $O(n)$ such terms. 
	
	In the second case, the only nonzero terms are
	$\tilde{h}(x_i, x_j)^4$ with $j\neq i$ (at most $O(n)$ such terms), and $\tilde{h}(x_i, x_j)^2\tilde{h}(x_i, x_k)^2$, with unique $k,j \neq i$ (at most $O(n^2)$). 
	
	Therefore, we get
	\begin{align*}
		\mathbb{E}\left(f_*(x_i) - \tilde{f}_n(x_i)\right)^4 \precsim \frac{n^2}{n^4} = \frac{1}{n^2},
	\end{align*}
	which implies that
	\begin{align*}
		\mathbb{E} \left[ C_2(\bm{X})^2 \right]  &\leq \sum_{i,j} \left[ \mathbb{E}\left(f_*(x_i) - \tilde{f}_n(x_i)\right)^4 \right]^{1/2} \left[ \mathbb{E}\left(f_*(x_j) - \tilde{f}_n(x_j)\right)^4 \right]^{1/2} \\
		 &\precsim n^2 \frac{1}{n^2} = 1.
	\end{align*}
	By the Chebyshev inequality, again we have the desired bound. 
	
\end{proof}

\subsection{Proofs in Section~\ref{sec:NN}}
\begin{proof}[Proof of Corollary~\ref{coro:ntk}]
Here we only give a sketch of the steps to avoid repetitions.
First, make approximation of the given kernel by a weighted sum of polynomial-type inner-product kernels. Note that 
\begin{equation*}
	\cos\angle (\tilde x, \tilde x') =\frac{x^\top x'}{\| x\|\| x'\|},
\end{equation*}
which is an inner product kernel divided by $\|\tilde x\|\|\tilde x'\|$.
Therefore, we write kernel $k$ as
\begin{equation*}
	k(x,x')/d= \sum\limits_{\iota=0}^\infty  \alpha_\iota \left\| \frac{1}{\sqrt{d}}  x \right\|^{1- \iota} \left\| \frac{1}{\sqrt{d}}  x'  \right\|^{1 - \iota} \left( \frac{x^\top x'}{d} \right)^{\iota}
\end{equation*}
For a constant $\iota_0$ large enough, we would have $k^{[\le \iota_0]}$ so close to $k$ so that
\begin{equation*}
	\left| \|k(\bm{X} , \bm{X} )^{-1}k(\bm{X} ,x)\|^2 -  \|k^{[\le \iota_0]}(\bm{X} , \bm{X} )^{-1}k^{[\le \iota_0]}(\bm{X} ,x)\|^2\right|
\end{equation*}
is of the order $ne^{-\Omega(\iota_0)}$. Then we need only to upper bound
\begin{equation*}
	\mathbb{E}_{x}\|k^{[\le \iota_0]}(\bm{X} , \bm{X} )^{-1}k^{[\le \iota_0]}(\bm{X} ,x)\|^2.
\end{equation*}
Define kernels $h_i$ as
\begin{equation*}
	h_{i}(x,x'):= \left( \frac{x^\top x'}{d} \right)^{i},
\end{equation*}
then we have
\begin{equation*}
	k(x,x')=\sum\limits_{i=0}^\infty  \alpha_{i} \left\| \frac{1}{\sqrt{d}}  x \right\|^{1- \iota} \left\| \frac{1}{\sqrt{d}}  x'  \right\|^{1 - \iota} h_{i}(x,x').
\end{equation*}

Define a diagonal matrix
	\begin{equation*}
		A:=\text{diag}(\| x_1 /\sqrt{d} \|,\cdots,\| x_n /\sqrt{d}\|).
	\end{equation*}
	Then we have
	\begin{equation*}
		k^{[\iota]}(\bm{X}, \bm{X} )=\alpha_\iota A^{1- \iota}h_\iota(\bm{X} )A^{1- \iota}
	\end{equation*}
	and
	\begin{equation*}
		k^{[\iota]}(\bm{X} ,x)=\alpha_i A^{1- \iota}h_ \iota(\bm{X} ,x)\| x/\sqrt{d}\|^{1- \iota}.
	\end{equation*}
	Now we have
	\begin{equation}
		k^{[\iota]}(\bm{X} )^{-1}k^{[\iota]}(\bm{X} ,x)=A^{\iota-1}h_ \iota(\bm{X} ,x)\|x /\sqrt{d}\|^{1- \iota}.
	\end{equation}
	Note that w.h.p., we have
	\begin{equation}
		A\asymp \mathbf{I}_n, \quad \| x/\sqrt{d}\| \asymp 1,
	\end{equation}
	which implies $\|x /\sqrt{d} \|^{\iota} \asymp 1$ for all $\iota \leq \iota_0 = o(\sqrt{d}/\log n)$ uniformly.

Therefore as in Proposition~\ref{prop:key}, we proceed with
\begin{equation*}
	\begin{split}
		&\mathbb{E}_{x}\|k^{[\le \iota_0]}(\bm{X} , \bm{X} )^{-1}k^{[\le \iota_0]}(\bm{X} ,x)\|^2\\
		\lesssim & \sum_{i=0}^{\iota}\mathbb{E}_{x}\|k^{[\le \iota_0]}(\bm{X} , \bm{X} )^{-1} k^{[i]}(\bm{X} ,x)\|^2 +  \mathbb{E}_{x}\|k^{[\le \iota_0]}(\bm{X} , \bm{X} )^{-1} \sum_{i=\iota+1}^{\iota_0} k^{[i]}(\bm{X} ,x)\|^2 \\
		\lesssim & \sum_{i=0}^{\iota}\mathbb{E}_{x}\|k^{[\le \iota_0]}(\bm{X} , \bm{X} )^{-1}k^{[i]}(\bm{X} ,x)\|^2 + \| k^{[\le \iota_0]}(\bm{X} , \bm{X} )^{-1} \|_{\rm op}^2  \cdot \mathbb{E}_{x}\| \sum_{i=\iota+1}^{\iota_0} k^{[i]}(\bm{X} ,x)\|^2\\
		\le&\sum_{i=0}^{\iota}\mathbb{E}_{x}\|k^{[i]}(\bm{X} , \bm{X} )^{+}\|_{\rm op}^2  \cdot \| k^{[i]}(\bm{X} ,x)\|^2 + O(1) \frac{n}{d^{\iota+1}} \\
		\lesssim&\sum_{i=0}^{\iota}\|h_{i}(\bm{X}, \bm{X} )^{+}\|_{\rm op}^2 \cdot \mathbb{E}_{x} \|h_{i}(\bm{X} ,x)\|^2 + \frac{n}{d^{\iota+1}}\\
		\lesssim& \sum_{i=0}^{\iota} (n d^{-i})^{-2} \frac{n}{d^i} + \frac{n}{d^{\iota+1}} \\
		\lesssim& \frac{d^{\iota}}{n} + \frac{n}{d^{\iota+1}}.
	\end{split}
\end{equation*}

Now adding back the truncated term, for $\iota_0 = o(\sqrt{d}/\log n)$, we have
\begin{equation}
	\mathbb{E}_{x\sim \mathcal{P}_X}\|k(\bm{X} , \bm{X} )^{-1}k(\bm{X} ,x)\|^2\lesssim  \frac{d^ \iota}{n} + \iota_0 \frac{n}{d^{\iota+1}}+ n e^{- \Omega(\iota_0)}\enspace.
\end{equation}
Take $\iota_0=C(\log d + \log n)\gg \iota$ with $C$ large enough, it is clear that $\iota_0 \ll \sqrt{d}/\log n$, we would have
\begin{equation*}
	\mathbb{E}_{x}\|k(\bm{X} , \bm{X} )^{-1}k(\bm{X} ,x)\|^2\lesssim  \frac{d^ \iota}{n} + \iota_0 \frac{n}{d^{\iota+1}}\enspace.
\end{equation*}

\end{proof}